\documentclass[letterpaper,11pt]{amsart}
\usepackage{latexsym,array,delarray,amsthm,amssymb,epsfig,amsmath,kbordermatrix,blkarray,tikz,algorithmic}
\usepackage{url}
\usepackage{tkz-euclide,subfigure}
\usepackage{longtable}
\usetikzlibrary{decorations.markings}
\usetikzlibrary{arrows}

\theoremstyle{plain}
\newtheorem{thm}{Theorem}[section]
\newtheorem{lemma}[thm]{Lemma}
\newtheorem{prop}[thm]{Proposition}
\newtheorem{cor}[thm]{Corollary}

\newtheorem*{thm*}{Theorem}
\newtheorem*{lemma*}{Lemma}
\newtheorem*{prop*}{Proposition}
\newtheorem*{cor*}{Corollary}
\newtheorem*{conj*}{Conjecture}
\newtheorem*{rep@theorem}{\rep@title}
\newcommand{\newreptheorem}[2]{%
\newenvironment{rep#1}[1]{%
 \def\rep@title{#2 \ref{##1}}%
 \begin{rep@theorem}}%
 {\end{rep@theorem}}}
 \newreptheorem{theorem}{Theorem}

\theoremstyle{definition}
\newtheorem{defn}[thm]{Definition}
\newtheorem{ex}[thm]{Example}

\newtheorem{alg}[thm]{Algorithm}

\theoremstyle{remark}
\newtheorem*{rmk}{Remark}

\newcommand{\nn}{\mathbb{N}}

\newcommand{\rr}{\mathbb{R}}
\newcommand{\cc}{\mathbb{C}}

\newcommand{\cm}{\mathcal{M}}

\newcommand{\Con}{\mathrm{Con}}
\newcommand{\ind}{\mbox{$\perp \kern-5.5pt \perp$}}

\newcommand{\supp}{\mathrm{supp}}
\newcommand{\spou}{\mathrm{Sp}}
\newcommand{\rank}{\mathrm{rank}}
\renewcommand{\kbldelim}{(}
\renewcommand{\kbrdelim}{)}
\newcommand{\imps}{\;\rotatebox[origin = c]{270}{\ind}_{P}\;}
\newcommand{\impa}{\;\rotatebox[origin = c]{270}{\mbox{$\perp$}}_{A}\;}
\newcommand{\cl}{\mathrm{cl}}

\begin{document}
\title{Marginal Independence and Partial Set Partitions}
\author{Francisco Ponce-Carri\'on and Seth Sullivant}
\address{Department of Mathematics \\
North Carolina State University, Raleigh, NC, 27695}
\email{fmponcec@ncsu.edu}
\email{smsulli2@ncsu.edu}

\maketitle
\begin{abstract}
We establish a bijection between marginal independence models on $n$ random variables
and split closed order ideals in the poset of partial set partitions.  
We also establish that every discrete marginal independence model is toric
in cdf coordinates.  This generalizes results of Boege, Petrov\'ic, and Sturmfels \cite{boege}
and Drton and Richardson \cite{drton}, and provides a unified framework for
discussing marginal independence models. Additionally, we provide an axiomatic characterization of marginal independence and we show that our set of axioms are sound and complete in the set of probability distributions. This follows the work of Geiger, Paz and Pearl \cite{geiger} who provided an analogous characterization of independence for statements involving 2 sets of random variables.
\end{abstract}

\section{Introduction}

Marginal independence models are often studied by associating a combinatorial structure to describe the independence relations between random variables. 
For instance, Drton and Richardson \cite{drton} used bidirected graphs to
develop a family of marginal independence models for discrete random variables,
and Boege, Petrov\'ic, and Sturmfels \cite{boege} used simplicial complexes.  
The marginal independence models of \cite{drton}  fit more broadly
into the class of graphical models  \cite{lauritzen, sullivant}. 
Each of these structures describes a very different family of models, 
with very little overlap between them.
In both cases it had been shown that when we restrict to discrete random variables, 
we obtain toric models after a linear change of coordinates. 
However, there are collections of marginal independence statements that cannot be expressed
either as graphs or simplicial complexes. This motivates the proposal of a more general combinatorial structure which we use in our treatment of marginal independence.

In this paper, we study marginal independence using the poset of partial set partitions, 
$P\Pi_{n}$. In Section \ref{sec:poset}, we begin by studying the combinatorial properties of $P\Pi_n$. We show its interplay with probability distributions on $n$ random variables. In this process, the class of split closed order ideals arises and we
give a combinatorial description of marginal independence models 
on $n$ random variables by showing a correspondence between the 
models and split closed order ideals of $P\Pi_{n,2}$, the subposet of
$P \Pi_n$ where every set partition has $2$ or more blocks. Explicitly, we have:

\begin{thm*}
	Marginal independence models on $n$ random variables are in bijective correspondence to split closed order ideals  in $P\Pi_{n,2}$.
\end{thm*}

In Section \ref{sec:axiom}, we give an axiomatic characterization of marginal independence, following the results by Geiger, Paz and Pearl \cite{geiger}. We provide a set of axioms on the poset of partial set partitions and prove that they are complete and sound in the set of probability distributions. Previous work on the axiomatic characterization of marginal independence used as a base structure statements involving only two disjoint sets, of the form $A\ind B$. Our formulation allows for total independence of more than two sets of random variables, which results in a more intuitive set of axioms than appears in \cite{geiger}.

In Section \ref{sec:discretemods}, we then restrict to discrete random variables and show that the resulting varieties are toric and smooth after a linear change of coordinates. 
In \cite{boege,drton,sullivant}, the proof that the mentioned family of
marginal independence models are toric uses a linear change of coordinates 
into \emph{M\"obius coordinates}, by using the M\"obius transformation on
an appropriately defined poset.
In this paper, we describe a change of coordinates 
using the cumulative distribution function, which is classically used to describe marginal independence.  This description has the advantage of tying the results 
to the classical description of marginal independence. We obtain the following result from this section:

\begin{thm*}
For a split closed order ideal $\mathcal{C}\subseteq P\Pi_{n,2}$, the affine variety that defines the model is a smooth toric variety in cdf coordinates.
\end{thm*}

This is a key result used in the proof of the correspondence between split closed order ideals and marginal independence models. We then move from the affine description of our models to the projective description by showing a parametrization of the homogeneous ideal corresponding to the model.

Finally, in Section \ref{sec:graphsimp} we discuss how marginal independence models associated to graphs and simplicial complexes can fit in the framework proposed in this paper. We then compare both classes of models and give an explicit characterization of the models that can be described using both simplicial complexes and graphs. We also count the number of distinct models 
in these three classes for up to $4$ random variables.


\section{Marginal Independence and the Poset of Partial Set Partitions} \label{sec:poset}

In this section we discuss general marginal independence models,
for completely general types of random variables.  
Marginal independence is naturally described via constraints on 
the joint cumulative distribution function of a collection of random variables.
We will also see a combinatorial structure which can be used to
capture all marginal independence models. 
This structure is called the poset of partial set partitions, 
and we will show how certain order ideals of this poset 
correspond to marginal independence models. 
We achieve this by introducing a closure operation on 
order ideals which is based on translating certain implications  
of marginal independence statements to the language of this poset. 

\begin{defn}
Let $X_1, \ldots, X_n$ be random variables in $\mathbb{R}$.
The joint cumulative distribution function (cdf) of $X_1, \ldots, X_n$, 
is the function defined on $(\rr\cup \{\pm \infty\})^n$ defined by the rule
\[
F(a_1, \ldots, a_n)   =  {\rm P}( X_1 \leq a_1, \ldots, X_n \leq a_n ).
\]
\end{defn}

Note that for a function $F$ to be a joint cumulative distribution function
it must satisfy a few properties.
\begin{itemize}
\item  For all $i$, $ \lim_{x_i \rightarrow -\infty}  F(x_1, \ldots, x_n) = 0$
\item  $\lim_{x_1, \ldots, x_n \rightarrow \infty} F(x_1, \ldots, x_n) = 1$.
\item  $F$ should be nondecreasing in each coordinate:  That is, if $a \leq b$ coordinate-wise,
then $F(a) \leq F(b)$.  
\item  $F$ should be right continuous in each coordinate:  That is, for each $i$,
and all $a_1, \ldots, a_n$,  
\[
\lim_{x_i \rightarrow a_i^+}  F(a_1, \ldots, a_{i-1}, x_i, a_{i+1}, \ldots, a_n )  =  
F(a_1, \ldots, a_n)
\]
\item  For all $a, b$,  $F(\max(a,b)) + F(\min(a,b))  \geq F(a) + F(b)$.  
\end{itemize}

These conditions are not exhaustive; that is, there are functions that satisfy all of these
properties but are not cdfs of a collection of random variables.  However,
we do not need to know a precise characterization of functions which are
joint cdfs to derive our results.

Note that the joint cdf contains all the marginal cdfs as well.
For example, with three random variables, we have
\[
F(a_1, a_2, \infty)  =  {\rm P}(X_1 \leq a_1, X_2 \leq a_2,  X_3 \leq \infty) 
=  {\rm P}(X_1 \leq a_1, X_2 \leq a_2)  
\]

Marginal independence of collections of random variables can be expressed in terms
of the joint cdf.    Let $X = (X_1, \ldots, X_n)$
denote the random vector, and $X_A = (x_i :  i \in A)$ be a subvector.

\begin{defn}
Let $A, B \subseteq [n]$, be disjoint subsets.  We say that $X_A$ and $X_B$ are
\emph{marginally independent} (denoted $X_A \ind X_B$, or just $A \ind B$ for short),
if
\[
F(x)  =  F(y) F(z)
\]
for all $x$ such that $x_i = \infty$ for all $i \in [n] \setminus (A \cup B)$ and where
\[
y_i  =  \begin{cases}
x_i  &  \mbox{if }  i \in A  \\
\infty &  \mbox {otherwise}
\end{cases}  \quad
\quad \quad
\mbox{ and }
\quad
\quad \quad
z_i  =  \begin{cases}
x_i  &  \mbox{if }  i \in B  \\
\infty &  \mbox {otherwise}
\end{cases} 
\]
\end{defn}

\begin{ex}
Suppose $n = 4$, and consider the marginal independence statement $X_1 \ind (X_2, X_4)$
(or $\{1 \} \ind \{2, 4 \}$, or $1 \ind 24$).  This translates into the condition:
\[
F(x_1, x_2, \infty, x_4)  =  F(x_1, \infty, \infty, \infty) F(\infty, x_2, \infty, x_4),\qquad \forall\; x_1,x_2,x_4 \in \rr \cup \{\infty\}
\]
Said another way, we have, for all $x_1, x_2, x_4$,
\[
{\rm P}(X_1 \leq x_1, X_2 \leq x_2, X_4 \leq x_4)  =  
{\rm P}(X_1 \leq x_1) 
{\rm P}( X_2 \leq x_2, X_4 \leq x_4).
\]
\end{ex}
More generally, we can define marginal independence for a collection of sets.

\begin{defn}\label{def:margmult}
Let $A_1, \ldots, A_k \subseteq [n]$ with $A_i \cap A_j  = \emptyset$ for all $i \neq j$.
Then $X_{A_1} \ind \cdots \ind X_{A_k}$ if, for all $x$ with
$x_i = \infty$ for all $i \in [n] \setminus (A_1 \cup \cdots \cup A_k)$, we have
\[
F(x) =  F(y^{(1)})  \cdots  F(y^{(k)})
\]
where
\[
y^{(j)}_i  =   \begin{cases}  
x_i  & \mbox{if }  i \in A_j \\ 
\infty  &  \mbox{otherwise}.
\end{cases}
\]
\end{defn}

Now that we have the notion of generalized marginal independence, we can discuss
marginal independence models.  A marginal independence model is a collection
of marginal independence statements and the set of probability distributions that
are compatible with those statements.  We will describe a completely general combinatorial
structure for representing all such marginal independence models.  This depends on 
the structure of a certain poset called the poset of partial set partitions.

\begin{defn}	
A \emph{partial set partition} of $[n]$ is an unordered list of nonempty sets 
$\pi = \pi_1 | \cdots | \pi_k$
such that each $\pi_i \subseteq [n]$, and $\pi_i \cap \pi_j = \emptyset$ for all $i \neq j$.
The set of all partial set partitions of $[n]$ is denoted $P\Pi_n$.   A set $\pi_i$
in the partial set partition $\pi$ is called a \emph{block}. The ground set, $|\pi|$ of $\pi$ is defined as $|\pi| = \cup_{i =1}^k \pi_i$.
The set of partial set partitions where each $\pi$ has at least $2$ blocks is denoted
$P\Pi_{n, 2}$.
\end{defn}

Note the elements of $P\Pi_{n, 2}$ correspond to marginal independence statements.
Specifically $\pi = \pi_1| \cdots | \pi_k$ corresponds to the independence statement
$X_{\pi_1} \ind \cdots  \ind X_{\pi_k}$.

We want to turn $P\Pi_n$ into a poset that is compatible with marginal independence implications.
We will explore some implications between marginal independence statements that will 
lead naturally to a poset structure on $P\Pi_n$.

\begin{lemma}\label{lem:marginalimpl}
If a cdf $F$ satisfies $A_1\ind \ldots \ind A_k$ then,
 \begin{enumerate}
	\item  $F$ satisfies $\cup_{s\in S} A_s \ind \cup_{t\in T}A_t$ for all $S,T \subseteq [k]$ such that $S\cap T = \emptyset$.
	\item $F$ satisfies $B_1\ind \ldots \ind B_k$ where $B_i \subseteq A_i$ for all $i$.
\end{enumerate}
\end{lemma}
	\begin{proof}
		First, we will show that for $S,T \subseteq [k]$ such that $S\cap T = \emptyset$, if $F$ satisfies $A_1\ind \ldots \ind A_k$, then it satisfies $\cup_{s\in S} A_s \ind \cup_{t\in T}A_t$. Let $x$ be such that $x_i =\infty$ if $i\in [n]\setminus  \cup_{s\in S} A_s$, in particular, $x_i = \infty$ if $i \in [n] \setminus (A_1 \cup \cdots \cup A_k)$ as well. Thus,
\[
F(x) =  F(y^{(1)})  \cdots  F(y^{(k)})
\]
where
\[
y^{(j)}_i  =   \begin{cases}  
x_i  & \mbox{if }  i \in A_j \mbox  { and } j\in S \\ 
\infty  &  \mbox{otherwise}.
\end{cases}
\]
Hence, $$F(x) = \prod_{s\in S} F\left(y^{(s)}\right).$$
We get an analogous result for $\cup_{t\in T}A_t$.

Now, let $x$ be such that $x_i = \infty $ if $i \in [n]\setminus \cup_{t\in S\cup T}A_t$. Then, it holds that 
\[
F(x) =  F(y^{(1)})  \cdots  F(y^{(k)})
\]
where
\[
y^{(j)}_i  =   \begin{cases}  
x_i  & \mbox{if }  i \in A_j \mbox  { and } j\in S\cup T \\ 
\infty  &  \mbox{otherwise}.
\end{cases}
\]
Hence, \begin{equation*}
	\begin{aligned}
		F(x) &= \prod_{s\in S} F\left(y^{(s)}\right) \prod_{t\in T} F\left(y^{(t)}\right)\\
		&= F(w)F(z)
	\end{aligned}
\end{equation*}
where 
\[
w_i  =  \begin{cases}
x_i  &  \mbox{if }  i \in \cup_{s\in S}A_s  \\
\infty &  \mbox {otherwise}
\end{cases}  \quad
\quad \quad
\mbox{ and }
\quad
\quad \quad
z_i  =  \begin{cases}
x_i  &  \mbox{if }  i \in \cup_{t\in T}A_t  \\
\infty &  \mbox {otherwise}
\end{cases} 
\]

Now we will show that if $F$ satisfies $A_1\ind \ldots \ind A_k$, then it satisfies $B_1\ind \ldots \ind B_k$ whenever $B_i \subseteq A_i$ for all $i$. Assume $F$ satisfies $A_1\ind \ldots \ind A_k$. Let $x$ be such that $x_i = \infty$ for all $i\in [n] \setminus (B_1\cup \ldots \cup B_k)$, then it follows that 
 \[
F(x) =  F(y^{(1)})  \cdots  F(y^{(k)})
\]
where
\[
y^{(j)}_i  =   \begin{cases}  
x_i  & \mbox{if }  i \in B_j \subseteq A_j \\ 
\infty  &  \mbox{otherwise}.
\end{cases}
\]
which proves $F$ also satisfies $B_1\ind \ldots \ind B_k$.
\end{proof}
\begin{prop}\label{prop:marginalimplication}
Suppose that the joint cdf $F$ satisfies the marginal independence statement $A_1 \ind \cdots \ind A_k$.  Let $B_1 \ind \cdots \ind B_l$ be another marginal independence statement such that
for each $i\in [l]$, there is a set $I_i \subseteq [k]$ such that $I_i \cap I_j = \emptyset$ for all $i \neq j$, and $B_i \subseteq \cup_{t \in I_i} A_t$.
Then $F$ also satisfies $B_1 \ind \cdots \ind B_l$.
\end{prop}

\begin{proof}
Assume $F$ satisfies the marginal independence statement $A_1 \ind \cdots \ind A_k$.  Let $B_1 \ind \cdots \ind B_l$ be another marginal independence statement such that
for each $i\in [l]$, there is a set $I_i \subseteq [k]$ such that $I_i \cap I_j = \emptyset$ for all $i \neq j$, and $B_i \subseteq \cup_{t \in I_i} A_t$. By Lemma \ref{lem:marginalimpl} (1), we know that $F$ satisfies $\cup_{t \in I_1} A_t \ind \cdots \ind \cup_{t \in I_l} A_t$. Finally, $F$ satisfies $B_1\ind \cdots \ind B_l$ by Lemma \ref{lem:marginalimpl} (2), since $B_i \subseteq \cup_{t \in I_i} A_t$ for all $i \in [l]$.
\end{proof}

This yields a partial order on $P\Pi_{n}$ by the same rule as in Proposition \ref{prop:marginalimplication}.

\begin{defn}
Let $\pi = \pi_1| \cdots | \pi_k$ and $\tau = \tau_1 | \cdots | \tau_l$ be two
partial set partitions.
Define a partial order on $P\Pi_n$ by declaring $\pi > \tau$ if 
for each $i \in [l]$, there is a set $I_i \in [k]$ such that $I_i \cap I_J = \emptyset$ for
all $i \neq j$, and $\tau_i \subseteq \cup_{t \in I_i}  \pi_t$.
\end{defn}

Note that the covering relations in $P\Pi_n$ are of two types.
We have that $\pi$ covers $\tau$ if either
\begin{enumerate}
\item  $\pi$ can be obtained from $\tau$ by adding
a single element to one of the blocks of $\tau$, or
\item  $\pi$ can be obtained from $\tau$ by splitting  a block of $\tau$ into two
blocks.
\end{enumerate}

\begin{prop}\label{prop:leq}
	In $P\Pi_{n}$, $\pi < \tau$ if and only if $S=|\pi| \subseteq |\tau|$ and $\pi \succ_{S} \tau\cap S$, where $\succ_{S}$ is the standard order in the lattice of set partitions of $S$ and $\tau\cap S = \tau_1\cap S|\cdots|\tau_l\cap S$.
\end{prop}

	\begin{proof}
		This follows directly from the covering relations.
	\end{proof}

\begin{rmk}
Proposition \ref{prop:leq} implies that $\pi < \tau$ if and only if every part of $\tau \cap |\pi|$ is in a part of $\tau$. Additionally, note that $P\Pi_n$ is not, in general, a lattice.  Indeed, the elements $1|2$ and $3|4$
have two least upper bounds, namely $13|24$ and $14|23$.  This shows a key difference with the poset of set partitions $\Pi_n$, which is well known to be a lattice \cite{stanley}.
\end{rmk}
On the other hand, $P\Pi_n$ is a graded poset, with 
\[
\rho(\pi_1|\cdots | \pi_k) =\begin{cases}
	0 & \mbox{if } \pi = \emptyset\\
	\#( \cup_{i} \pi_i) + k -1 & \mbox{otherwise}
\end{cases} 
\]  
This formula follows by counting the number of covering relations to get from the bottom
element $\emptyset$ to a partial set partition $\pi$. Furthermore, we can restate the rank function as follows:
\[
\rho(\pi_1|\cdots | \pi_k) = \# \mbox{ of symbols in $\pi$} = 
\#( \cup_{i} \pi_i) + \# \mbox{ of bars}.\]

\begin{rmk}
Note that $P\Pi_n$ contains the Boolean lattice $B_n$ as a subposet.  The
lattice of set partitions $\Pi_n$ is not a subposet of $P\Pi_n$,
but rather its opposite $\Pi_n^{opp}$ is a subposet. In fact, $P\Pi_n$ contains the opposite of the lattice of set partitions of any $S\subseteq [n]$.
\end{rmk}  

\begin{center}
	\begin{figure}
	
	\subfigure[$P\Pi_{3}$]{
		\begin{tikzpicture}[scale=.5]

\node (e) at (0,-5) {$\emptyset$};

\node (1) at (-2,-2.5) {$1$};
\node (2) at (0,-2.5) {$2$};
\node (3) at (2,-2.5) {$3$};

\node (12) at (-2,0) {$12$};
\node (13) at (0,0) {$13$};
\node (23) at (2,0) {$23$};
  
  \node (1p2) at (-3,2.5) {$1|2$};
  \node (2p3) at (3,2.5) {$2|3$};
  \node (1p3) at (-1,2.5) {$1|3$};
  \node (123) at (1,2.5) {$123$};
 
  \node (1p23) at (-2,5.5) {$1|23$};
  \node (12p3) at (2,5.5) {$12|3$};
  \node (13p2) at (0,5.5) {$13|2$};
  
  \node (1p2p3) at (0,8) {$1|2|3$};
  \draw (1p2p3)--(1p23)--(1p2);
  \draw (1p2p3)--(12p3)--(1p3);
  \draw (1p2p3)--(13p2)--(1p2);
  \draw (1p23)--(1p3);
  \draw (12p3)--(2p3);
  \draw (13p2)--(2p3);
  
  \draw (13p2) -- (123);
  \draw (12p3) -- (123);
  \draw (1p23) -- (123);
  
  \draw (1p2) -- (12);
  \draw (123) -- (12);
  
  \draw (1p3) -- (13);
  \draw (123) -- (13);
  
  \draw (2p3) -- (23);
  \draw (123) -- (23);
  
  \draw (12) -- (1);
  \draw (12) -- (2);
  \draw (13) -- (1);
  \draw (13) -- (3);
  \draw (23) -- (2);
  \draw (23) -- (3);
  
  \draw (1) -- (e);
  \draw (2) -- (e);
  \draw (3) -- (e);
\end{tikzpicture}
}
\subfigure[$P\Pi_{3,2}$]{

\begin{tikzpicture}[scale=.5]  
  \node (1p2) at (-2,2.5) {$1|2$};
  \node (2p3) at (2,2.5) {$2|3$};
  \node (1p3) at (0,2.5) {$1|3$};
 
  \node (1p23) at (-2,5) {$1|23$};
  \node (12p3) at (2,5) {$12|3$};
  \node (13p2) at (0,5) {$13|2$};
  
  \node (1p2p3) at (0,7.5) {$1|2|3$};
  \draw (1p2p3)--(1p23)--(1p2);
  \draw (1p2p3)--(12p3)--(1p3);
  \draw (1p2p3)--(13p2)--(1p2);
  \draw (1p23)--(1p3);
  \draw (12p3)--(2p3);
  \draw (13p2)--(2p3);
 
\end{tikzpicture}

}
\caption{Hasse diagrams of $P\Pi_{3}$ and $P\Pi_{3,2}$}
\end{figure}
\end{center}

For studying marginal independence models, it is more natural to consider the 
poset $P\Pi_{n,2}$ of partial set partitions where each partition has at least two blocks.
We will see that marginal independence models correspond to certain types of order
ideals in $P\Pi_{n,2}$.

\begin{defn}\label{genmargmodel}
Let $\pi^{(1)},\ldots, \pi^{(k)} \subseteq P\Pi_{n,2}$, then 
$$\mathcal{C} = \left\{\pi_1^{(1)}|\cdots|\pi_{i_1}^{(1)},\ldots, \pi_1^{(k)}|\cdots|\pi_{i_k}^{(k)}\right\}$$
or, equivalently,
$$\mathcal{C} = \left\{X_{\pi_1^{(1)}}\ind \ldots \ind X_{\pi_{i_1}^{(1)}},\ldots, X_{\pi_1^{(k)}}\ind \ldots \ind X_{\pi_{i_k}^{(k)}}\right\}$$
is a list of independence statements on random variables $X_1,\ldots,X_n$. A \emph{marginal independence model}, $\mathcal{M}_\mathcal{C}$, on these random variables is the set of distributions $P$ that satisfy all the statements in $\mathcal{C}$. 

\end{defn}
\begin{defn}\label{def:split1}
	 Let $\pi = \pi_1|\cdots | \pi_k \in P\Pi_{n,2}$ and $\tau = \tau_1 | \tau_2 \in P\Pi_{n,2}$. We say $\tau$ \textit{splits} $\pi$ if $\pi_i = \tau_1 \cup \tau_2$ for some $i$. We denote the \textit{splitting of $\pi$ by $\tau$} with $\pi^\tau = \pi_1|\cdots|\pi_{i-1}|\tau_1 | \tau_2|\pi_{i+1}| \cdots  | \pi_k$.
\end{defn}

\begin{defn}
	Let $\mathcal{C}$ be an order ideal of $P\Pi_{n,2}$. We say $\mathcal{C}$ is \textit{split closed} if $\pi^\tau \in \mathcal{C}$ for all $\pi,\tau\in  \mathcal{C}$ such that $\tau$ splits $\pi$.
\end{defn}

\begin{defn}
	The \textit{split closure} $\overline{\mathcal{C}}$ of $\mathcal{C}$ is the smallest split closed order ideal containing $\mathcal{C}$.
\end{defn}

\begin{ex}\label{ex:split}
	Let $n=3$ and let $\mathcal{C}$ be the order ideal generated by $1|23$ and $2|3$ (see Figure \ref{fig:PPi3}). Then, the splitting $1|2|3$ must be in $\overline{\mathcal{C}}$ and therefore $\overline{\mathcal{C}} = P\Pi_{3,2}$. 
	\end{ex}
\begin{figure}\label{fig:PPi3}
		\begin{tikzpicture}[scale=.5]  
  \node[color=blue] (1p2) at (-2,2.5) {\color{blue}$\underline{1|2}$};
  \node (2p3) at (2,2.5) {\color{blue}$\underline{2|3}$};
  \node (1p3) at (0,2.5) {\color{blue}$\underline{1|3}$};
 
  \node (1p23) at (-2,5) {\color{blue}$\underline{1|23}$};
  \node (12p3) at (2,5) {$12|3$};
  \node (13p2) at (0,5) {$13|2$};
  
  \node (1p2p3) at (0,7.5) {$1|2|3$};
  \draw (1p2p3)--(1p23)--(1p2);
  \draw (1p2p3)--(12p3)--(1p3);
  \draw (1p2p3)--(13p2)--(1p2);
  \draw (1p23)--(1p3);
  \draw (12p3)--(2p3);
  \draw (13p2)--(2p3);
 
\end{tikzpicture}\caption{The order ideal generated by $1|23$ and $2|3$ from Example \ref{ex:split}}
	\end{figure}

\begin{prop}\label{prop:iffsplit}
	Let $\mathcal{C}$ be any subposet of $P\Pi_n$. A distribution $F$ satisfies the statements in $\mathcal{C}$ if and only if it satisfies the statements in $\overline{\mathcal{C}}$.
\end{prop}
\begin{proof}
	Let $F$ be the cdf of random variables $X_1,\ldots, X_n$ satisfying $\mathcal{C}$. It follows from Proposition \ref{prop:marginalimplication} that $F$ satisfies the independence relations in the order ideal generated by $\mathcal{C}$, we will denote this order ideal by $\tilde{\mathcal{C}}$. Hence, it remains to be shown that $F$ satisfies any splitting in $\tilde{\mathcal{C}}$. Let $\pi,\tau \in \tilde{\mathcal{C}}$, such that $\tau$ splits $\pi$, in particular, let $\tau$ split $\pi_l$ for some $l\in [k]$. Then, by definition of the statement $\pi$, we have that for any $x$ with $x_i=\infty$ for all $i\in [n]\setminus (\pi_1\cup \cdots \cup \pi_k)$,
	$$F(x) = F(y^{(1)})\cdots F(y^{(l)}) \cdots  F(y^{(k)})$$
	where 
	$$y^{(j)}_i  =   \begin{cases}  
x_i  & \mbox{if }  i \in \pi_j \\ 
\infty  &  \mbox{otherwise}.
\end{cases}$$
Furthermore, since $F$ satisfies $\tau$, we have that $F(y^{(l)}) = F(w^{(1)})F(w^{(2)})$ where $$w^{(j)}_i  =   \begin{cases}  
x_i  & \mbox{if }  i \in \tau_j \\ 
\infty  &  \mbox{otherwise.}
\end{cases}$$Putting both together we obtain 
$$F(x) = F(y^{(1)})\cdots F(y^{(l-1)}) F(w^{(1)})F(w^{(2)}) F(y^{(l+1)}) \cdots  F(y^{(k)}).$$
This shows $F$ satisfies the splitting and therefore it satisfies $\overline{\mathcal{C}}$. The converse follows from $\mathcal{C} \subseteq \overline{\mathcal{C}}$.
\end{proof}

\begin{rmk}
Proposition \ref{prop:iffsplit} allows us to make a precise connection between collections of marginal independence statements $\mathcal{C}$ and certain order ideals in $P\Pi_{n,2}$, namely, their split closure $\overline{\mathcal{C}}$.	
\end{rmk}

\begin{prop}\label{prop:meet}
	Let $\mathcal{C}$ be a a split closed order ideal. If $\pi,\mu \in \mathcal{C}$ are such that $S=|\pi| = |\mu|$ for some $S\subseteq [n]$, then $\pi \wedge \mu \in \mathcal{C}$ where $\wedge$ is the common refinement in the lattice of set partitions of $S$.
	\end{prop}	
	
	\begin{proof}
		Let $\pi,\mu \in \mathcal{C}$ be such that $S=|\pi| = |\mu|$ for some $S\subseteq [n]$. Observe that $\pi \wedge \mu = \tau$ where for each $i$, $\tau_i = \pi_r\cap \mu_t$ for some $r,t$. Now consider the partition $\gamma^i = \mu \cap \pi_i = \mu_1\cap \pi_i|\cdots | \mu_l \cap \pi_i$, by definition, either $\gamma^i = \pi_i \subseteq \mu_j$ for some $j$ or $\gamma^i \leq \mu$. Observe further that we can write $\tau = \gamma^1|\cdots|\gamma^k$. If for all $i$, we have that $\gamma^i = \pi_i$, then $\tau = \pi \in \mathcal{C}$. On the other hand, if $\gamma^i\leq \mu$, clearly $\gamma^i \in \mathcal{C}$ and since $\mathcal{C}$ is a split closed order ideal, so will $\pi^{\gamma^i}$. Thus, $\tau = \pi^{\gamma^1\cdots \gamma^k} \in \mathcal{C}$
	\end{proof}

	We finish this section stating a general theorem about marginal independence models that makes the correspondence between split closed order ideals and models precise. We prove this theorem in Section \ref{sec:discretemods}.
	
	\begin{thm}\label{thm:bijection}
	Marginal independence models $\mathcal{M}_\mathcal{C}$ on $n$ random variables are in bijective correspondence to split closed order ideals $\mathcal{C} \subseteq P\Pi_{n,2}$.
	\end{thm}
	
In this level of generality, it is not clear how to prove the correspondence yet. The forward direction follows from Proposition \ref{prop:iffsplit}. We have that for any marginal independence model, there exists a split closed order ideal by taking all implied marginal
independence statements (take the order ideal generated by the independence statements first and then include every possible splitting). 
For the other direction, it suffices to show the existence of 
probability distributions that satisfy exactly the independence statements 
in a given split closed order ideal, and no others. 
We will derive such result and present a full proof as a consequence of our study of marginal independence for
discrete random variables in Section \ref{sec:discretemods}.
\medskip
\section{Axiomatic characterization of independence}\label{sec:axiom}
In this section, we frame our results in terms of the axiomatic characterization of probabilistic marginal independence, following the results in \cite{geiger}. That paper provides a set of axioms that characterize independence using pairs of disjoint sets (in our case, partitions with only 2 blocks) and shows that the axioms are sound and complete. We show analogous results for our case. 

\begin{defn}
	Let $P$ be a class of probability distributions, $A$ be a set of axioms, $\sigma$ a statement and $\Sigma$ a set of statements. 
	
	\begin{enumerate}
	\item A statement $\sigma$ is logically implied by $\Sigma$, denoted $\Sigma \imps \sigma$ if every distribution in $P$ that satisfies $\Sigma$ also satisfies $\sigma$.
	\item We say $\sigma$ is derivable from $\Sigma$, denoted $\Sigma\impa\sigma$ if $\sigma$ can be derived from the axioms in $A$.
	\item $\cl_A(\Sigma)$ is the set of all statements derivable from $\Sigma$.
	\end{enumerate}

\end{defn}

\begin{defn}
	A set of axioms $A$ is \emph{sound} in $P$ if and only if for every statement $\sigma$ and every set of statements $\Sigma$
	$$\mbox{if $\Sigma \impa \sigma$ then } \Sigma \imps \sigma .$$
	The set $A$ is complete for $P$ if and only if
	$$\mbox{if $\Sigma \imps \sigma$ then } \Sigma \impa \sigma .$$
\end{defn}

\begin{prop}[\cite{geiger}]\label{prop:complete}
	A set of axioms $A$ is complete if and only if for every set of statements $\Sigma$ and every statement $\sigma \notin \cl_A(\Sigma)$, there exists a distribution $P_\sigma$ in $P$ that satisfies $\Sigma$ and does not satisfy $\sigma$.
\end{prop}

\begin{thm}
	Let $\Sigma \subseteq P\Pi_n$ , let $A$ be the following set of axioms:
	\begin{align*}
		\mbox{Decomposition } \qquad & \pi_1|\cdots | \pi_k \implies \tau_1 | \cdots |\tau_k, \forall \tau_i \subseteq \pi_i\\
		\mbox{Join } \qquad & \pi_1|\cdots | \pi_k \implies \pi_1 | \cdots|\hat{\pi_i}|\cdots | \hat{\pi_j} | \cdots  |\pi_k| \pi_i\cup \pi_j, \forall i,j \in [n]\\
		\mbox{Splitting } \qquad & \pi_1|\cdots | \pi_k \mbox{ and } \tau \mbox{ splits } \pi \implies \pi^\tau
	\end{align*}
	Then $A$ is sound and complete in the set of all probability distributions. Furthermore, $\cl_A(\Sigma) = \overline{\Sigma}$. 
\end{thm}

\begin{proof}
	First observe that $\cl_A(\Sigma)$ is an order ideal. This follows from the Decomposition and the Join axiom put together, as they both match the covering relations in $P\Pi_{n}$. Then it follows from the splitting axiom that $\cl_A(\Sigma) = \overline{\Sigma}$. Completeness of $A$ follows from Proposition \ref{prop:complete} and Theorem \ref{thm:bijection}. 
	
	The proof of soundness follows from Proposition \ref{prop:iffsplit} and the observation that $\cl_A(\Sigma) = \overline{\Sigma}$.
\end{proof}

In contrast, the work done by Geiger, Paz and Pearl in \cite{geiger} uses the base structure of pairs of disjoint subsets with analogous results for soundness and completeness:

\begin{thm}[\cite{geiger}]
	Let $R,S,T \subseteq [n]$ be disjoint. Let $B$ be the set of axioms
\begin{align*}
		\mbox{Trivial independence } \qquad & R \ind \emptyset\\
		\mbox{Symmetry } \qquad & R\ind S \implies S\ind R\\
		\mbox{Decomposition } \qquad & R\ind S\cup T \implies R\ind S\\
		\mbox{Mixing } \qquad & R\ind S \mbox{ and } R\cup S \ind T \implies R\ind S\cup T
	\end{align*}
Then $B$ is complete and sound in the set of all probability distributions.
\end{thm}

\begin{rmk}
Comparing the set of axioms $A$ to $B$ in their respective structures, we see that the Decomposition follows the same principle in both. $A$ does not need Symmetry, as set partitions are unordered. Furthermore, the Mixing axiom can be derived from the Join and Splitting axioms in $A$: Observe that in the language of set partitions, if $R|S$ and $R\cup S |T$, by the Splitting axiom, we have $R|S|T$, and then by the Join axiom, it follows that $R|S\cup T$.
\end{rmk}

In our approach, we take advantage of the poset structure to describe marginal independence, which allows for a more intuitive set of axioms. The Join and Splitting axioms come up naturally in the study of marginal independence. In addition, conveying independence information can be more complicated when we only allow pairs of subsets, in particular when we want to convey information about total independence of more than 2 subsets, whereas the poset structure allows us to find maximal generators of the split closed order ideals that work as representatives of the set of independence relations. 

We now tackle the following problem: given a set of statements $\mathcal{C} \subseteq P\Pi_{n,2}$, and a statement $\sigma$, we want to decide (in polynomial time) if $\sigma$ is implied by $\mathcal{C}$. Or, equivalently, if $\sigma \in \overline{\mathcal{C}}$. We begin by extending the operation of splitting to any 2 set partitions in $P\Pi_{n}$.
\begin{defn}\label{def:gensplit}Let $\pi,\gamma,\sigma \in P\Pi_n$.
\begin{enumerate}
	\item For $S \subseteq [n]$, define $\pi \cap S = \pi_1 \cap S | \cdots |\pi_k\cap S$, deleting any empty blocks.
	\item We define the splitting of $\pi$ by $\gamma$ as $\pi^\gamma$, where $\pi^\gamma$ is obtained from replacing $\pi_i$ with $\gamma\cap \pi_i$ if $|\gamma \cap \pi_i| = \pi_i$.
	\item Let $\mathcal{C} = \left\{\pi^{(1)}, \ldots, \pi^{(k)}\right\}\subseteq P\Pi_{n,2}$. We define the splitting map 
	\begin{align*}
		\mathcal{S}_\mathcal{C}: P\Pi_n &\rightarrow P\Pi_n\\
		\sigma &\mapsto \left(\left(\sigma^{\pi^{(1)}}\right)^{\cdot^{\cdot^\cdot}}\right)^{ \pi^{(k)}}.
	\end{align*}
	\item We define the splitting of $\sigma$ by $\mathcal{C}$ as $\sigma^\mathcal{C} = \mathcal{S}_\mathcal{C}^N(\sigma)$, where $N$ is such that $\mathcal{S}_\mathcal{C}^N(\sigma) = \mathcal{S}_\mathcal{C}^{N+1}(\sigma)$.
\end{enumerate}
\end{defn}
\begin{rmk}
In the case where $\gamma$ has 2 blocks and $|\gamma| = \pi_i$ for some $i$, Definition \ref{def:gensplit} (2) agrees with Definition \ref{def:split1}. Furthermore, if $\pi^\gamma = \pi$, we say $\gamma$ does not split $\pi$.

Observe that $N$ in Definition \ref{def:gensplit} (4) is guaranteed to exist, as $P\Pi_n$ is finite and $\sigma \leq \mathcal{S}_\mathcal{C}(\sigma)$. The fact that $\sigma^{\mathcal{C}}$ is well-defined follows from Proposition \ref{prop:corralg}.
\end{rmk}
\begin{ex}
	Let $\pi = 12|345|67$ and $\gamma = 1|4|356|7$. Then we have 
	\begin{align*}
		\gamma \cap \pi_1 &= 1, \; |\gamma \cap \pi_1| \neq \pi_1\\
		\gamma \cap \pi_2 &= 4|35, \; |\gamma \cap \pi_2| = \pi_2\\
		\gamma \cap \pi_3 &= 6|7, \; |\gamma \cap \pi_3| = \pi_3
	\end{align*}
	Hence, $\pi^\gamma$ is obtained by taking $\pi$ and replacing $\pi_2$ and $\pi_3$ with $\gamma \cap \pi_2$ and $\gamma \cap \pi_3$, respectively. $\pi^\gamma = 12|4|35|6|7$.
\end{ex}
\begin{prop}\label{prop:corralg}
	Let $\mathcal{C} = \left\{\pi^{(1)}, \ldots, \pi^{(k)}\right\}\subseteq P\Pi_{n,2}$, and let $D\subseteq|\pi^{(i)}| $ for some $i$. Then, $\gamma= D^\mathcal{C}$ is the unique maximal element in $\overline{\mathcal{C}}$ such that $|\gamma| = D$. 
\end{prop}
\begin{proof}
	Assume by contradiction that there exists $\tau \in \overline{\mathcal{C}}$ such that $\gamma <\tau$ and $|\tau| = D$. Without loss of generality, we may assume $\tau$ covers $\gamma$, and since their base sets are the same, the only option is that one of the blocks of $\gamma$ is split into 2 blocks in $\tau$. Since set partitions are unordered, we may assume $\gamma_1$ is split into $\tau_1|\tau_2$, implying that $\tau_1|\tau_2 \in \overline{\mathcal{C}}$. However, by construction, none of the $\pi^{(i)}$ or the elements below them split $\gamma_1$, since $\gamma_1^{\pi^{(1)}\cdots \pi^{(k)}} = \gamma_1$, so $\tau_1|\tau_2$ cannot be in $ \overline{\mathcal{C}}$.
	
	Now we will prove uniqueness. Let $\mu$ be another maximal element in $\overline{\mathcal{C}}$ such that $|\mu| = |\gamma| = D$. By Proposition \ref{prop:meet}, $\gamma \wedge \mu \in \overline{\mathcal{C}}$. Note that $\gamma \leq \gamma \wedge \mu$, and $\mu \leq \gamma \wedge \mu$, where $\wedge$ is taken from the lattice of set partitions of $D$. Since both $\gamma$ and $\mu$ are maximal elements, we must have $\gamma = \gamma \wedge \mu = \mu$.
\end{proof}

\begin{ex}
	Let $\mathcal{C} = \{ 3|4, 2|34 ,1|234\}$. Then we have:
	\begin{align*}
		\mathcal{S}_\mathcal{C}(1234) &= \left(\left(1234^{3|4}\right)^{2|34}\right)^{1|234}\\
		&= \left(1234^{2|34}\right)^{1|234}\\
		&= 1234^{1|234}\\
		&= 1|234
	\end{align*}
	Repeating this process yields $\mathcal{S}_\mathcal{C}^2(1234) = \mathcal{S}_\mathcal{C}(1|234) = 1|2|34$, and finally $\mathcal{S}_\mathcal{C}^3(1234) = \mathcal{S}_\mathcal{C}(1|2|34) = 1|2|3|4.$ Since that is the top element in $P\Pi_{4}$, we conclude $1234^{\mathcal{C}} = 1|2|3|4$.
\end{ex}

\begin{cor}\label{cor:corralg}
	Let $\mathcal{C} = \left\{\pi^{(1)}, \ldots, \pi^{(k)}\right\}\subseteq P\Pi_{n,2}$, and let $\sigma \in P\Pi_{n,2}$ be such that $|\sigma|\subseteq|\pi^{(i)}| $ for some $i$. Then, $\sigma \in \overline{\mathcal{C}}$ if and only if  $\sigma \leq  |\sigma|^\mathcal{C}$.
\end{cor}

\begin{alg}\label{alg:implication}\  Membership in split closure

\noindent \textbf{Input}: $\mathcal{C} = \left\{\pi^{(1)}, \ldots, \pi^{(k)}\right\}\subseteq P\Pi_{n,2}$, $\sigma \in P\Pi_{n,2}$.\\
\noindent \textbf{Output}: Whether $\sigma\in \overline{\mathcal{C}}$.
\begin{enumerate}
	\item If $|\sigma| \not\subseteq |\pi^{(i)}|$ for all $i$, then $\sigma \notin\overline{\mathcal{C}}$.
	\item If $|\sigma| \subseteq |\pi^{(i)}|$ for some $i$, compute $|\sigma|^\mathcal{C}$.
	\item If $\sigma \leq |\sigma|^\mathcal{C}$, then $\sigma\in \overline{\mathcal{C}}$. If not, then $\sigma\notin \overline{\mathcal{C}}$.
\end{enumerate}
\end{alg}

The correctness of Algorithm \ref{alg:implication} follows from Corollary \ref{cor:corralg} and the observation on step (1) that if the base set of $\sigma$ is not contained in any of the base sets of the elements of $\mathcal{C}$, it is not possible that $\sigma$ can be implied by $\mathcal{C}$, as split closure does not add elements to the base sets.

Now let us analyze the complexity of this algorithm.

\begin{prop}
The running time of Algorithm \ref{alg:implication} is $O(n^3 k)$.
\end{prop} 

\begin{proof}
Computing the splitting $\pi^\gamma$ requires computing $pq$ intersections and checking $p$ set equalities, where $p$ and $q$ are the number of blocks of $\pi$ and $\gamma$, respectively. 
Since $p$ and $q$ are both $\leq n$,  this step is $O(n^2)$.

Comparing $\pi$ to $\gamma$ requires checking at most $pq$ set containments.
Again, this yields  $O(n^2)$ for this step.  

Computing $\sigma^\mathcal{C}$ requires applying $\mathcal{S}_\mathcal{C}$ at most $2n$ times since the maximal length of any chain in $P\Pi_{n}$ is $2n$.   
Each application of $\mathcal{S}_\mathcal{C}$ requires $k$ splittings. 
Since each splitting requires at most $O(n^2)$ operations, we deduce that 
 Algorithm \ref{alg:implication} is $O(n^3k)$.
\end{proof}

\medskip
\section{Discrete marginal independence models}\label{sec:discretemods}

In this section, we will discuss marginal independence models restricting the class of random variables to have finitely many states. This allows us to translate the independence statements into polynomial constraints. Past work has been done on certain classes of these models where the marginal independence statements are given by graphs or simplicial complexes \cite{boege,drton}. However, there are marginal independence models that cannot be represented with either of these structures, which is why we present a more general approach using split closed order ideals from the set of partial set partitions. In previous work \cite{boege,drton, sullivant}, it has been shown these models are toric after a certain linear change of coordinates whose inverse could be found by applying M\"obius inversion on a specific poset. In our treatment of marginal independence, we will use the cumulative distribution function as our coordinate system. This allows us to express marginal independence constraints in a very natural way while preserving the structure of the inverse map as a M\"obius inversion on a poset isomorphic to the one used in \cite{boege}.

We begin by defining what a single marginal independence statement looks like in terms of the probability tensor and then we will generalize it further to define our models. For this, we need to first discuss some notation that will be used.

\begin{defn}
	For finite random variables $X_1,\ldots,X_n$, with respective state spaces $[r_i]$, $r_i \in \nn$, we use the shorthand $p_{i_1\ldots i_n}$ to denote $P(X_1 = i_1,\ldots X_n = i_n)$. Furthermore, when we discuss marginalizations of the probability tensor, we will allow $i_k = +$ to symbolize adding over all the states of the corresponding variable, i.e. if $i_k = +$, then 
	\begin{equation*}
	p_{i_1,\ldots,i_k,\ldots, i_n} =p_{i_1,\ldots,+,\ldots, i_n} = \sum_{j_k = 1}^{r_k}	p_{i_1,\ldots,j_k,\ldots, i_n}.
	\end{equation*}
	For an index vector $i=i_1\ldots i_n$, we define the \emph{support} of $i$ as $\supp(i) = \{l\in [n]\;|\; i_l \neq +\}$.
\end{defn}

\begin{defn}
	Let $X_1,\ldots, X_n$ be finite random variables, we say that for disjoint $A,B\subseteq [n]$, $X_A \ind X_B$ holds if and only if for a probability distribution $P$ indexed by $p_{i} = p_{i_1\ldots i_n}$ we have that
\begin{equation*}
	p_{i_1\ldots i_n}p_{j_1\ldots j_n} - p_{k_1\ldots k_n}p_{l_1\ldots l_n}=0,
\end{equation*}
for all $i, j, k, l$, with
 $\supp(i) = \supp(j)=\supp(k) = \supp(l) = A\cup B$, 
 $k_A = i_A$, $k_B = i_B$, $l_A = j_A$, and $l_B = j_B$. 

Equivalently, these constraints can be obtained from all the $2\times 2$ minors of the matrix resulting from marginalizing the probability tensor $P$ by summing out the indices not in $A\cup B$ and whose columns are indexed by the states of $X_A$ and the rows are indexed by the states of $X_B$, where $X_A$ and $X_B$ are seen as random vectors.
\end{defn}

\begin{defn}
	Let $P$ be a probability distribution on random variables $X_1,\ldots,X_n$ with respective state spaces $[r_1],\ldots,[r_n]$, with $r_i \in \nn$ for $i\in [n]$. Then, we define $\mathcal{R} = [r_1]\times \cdots \times [r_n]$.
\end{defn}

\begin{defn}
	We define the ring of polynomials in probability coordinates, $\rr[P]$, for the probability tensor of format $\mathcal{R} = [r_1]\times \ldots \times [r_n]$ as
	\begin{equation*}
		\rr[P] = \rr[p_i:\;i\in \mathcal{R}].
	\end{equation*}
\end{defn}

\begin{defn}\label{def:MIC}
	The \emph{marginal independence ideal} ($MI_\mathcal{C}$) associated to the split closed order ideal $\mathcal{C}$ on discrete random variables $X_1,\ldots,X_n$ is
	\begin{equation*}
		\begin{aligned}
			MI_{\mathcal{C}} = \left\langle 2\times 2 \mbox{ minors of all flattenings of } P_\pi:\; \pi \in \mathcal{C} \right\rangle,
		\end{aligned}
	\end{equation*}
	where $\pi = \pi_1|\cdots|\pi_k$ and $P_\pi$ is the $k$-way tensor which is obtained by first marginalizing the tensor $P$ by summing out the indices not in $|\pi|$, and then flattening the resulting tensor according to the set partition $\pi$.
	
	Alternatively, $MI_{\mathcal{C}}$ is the ideal resulting from imposing rank 1 conditions on the tensors $P_\pi$, for all $\pi \in \mathcal{C}$. This definition is a slight modification of the one in \cite{boege} to fit our class of models.
\end{defn}

\begin{ex}
	Let $X_1$ be a binary random variable and let $X_2, X_3$ be ternary randcom variables and
	consider the marginal independence statement $X_{2} \ind X_3$. Then the equations that define $MI_{2\ind 3}$ are
	\begin{equation*}
	\begin{aligned}
			p_{+i_2i_3}p_{+j_2j_3} - p_{+i_2j_3}p_{+j_2i_3} =& (p_{1i_2i_3}+p_{2i_2i_3})(p_{1j_2j_3}+p_{2j_2j_3})
			\\ &-(p_{1i_2j_3}+p_{2i_2j_3})(p_{1j_2i_3}+p_{2j_2i_3}) = 0,
	\end{aligned}
	\end{equation*}
	for all $i_2,i_3,j_2,j_3 \in [3]$. We can also obtain these equations from the $2\times 2$ minors of the matrix
	\begin{equation*}
		\begin{pmatrix}
			p_{+11} & p_{+12} & p_{+13}\\
			p_{+21} & p_{+22} & p_{+23}\\
			p_{+31} & p_{+32} & p_{+33}
		\end{pmatrix}.
	\end{equation*}
\end{ex}

\begin{ex}\label{ex:onest}
	Let $n=4$ and consider the model given by binary random variables such that $X_{1,2} \ind X_3$. Then, the equations that define $MI_{12\ind 3}$ are 
	\begin{equation*}
	\begin{aligned}
			p_{i_1i_2i_3+}p_{j_1j_2j_3+} - p_{i_1i_2j_3+}p_{j_1j_2i_3+} =& (p_{i_1i_2i_31}+p_{i_1i_2i_32})(p_{j_1j_2j_31}+p_{j_1j_2j_32}) \\ &- (p_{i_1i_2j_31}+p_{i_1i_2j_32})(p_{j_1j_2i_31}+p_{j_1j_2i_32}) = 0,
	\end{aligned}
	\end{equation*}
	for all $i_1,i_2,i_3,j_1,j_2,j_3 \in [2]$. These equations can again be obtained from the $2\times 2$ minors of the matrix
	\begin{equation*}
		\begin{pmatrix}
			p_{111+} & p_{112+}\\
			p_{121+} & p_{122+}\\
			p_{211+} & p_{212+}\\
			p_{221+} & p_{222+}\\
		\end{pmatrix}.
	\end{equation*}	
\end{ex}

The ideal $MI_{\mathcal{C}}$ tends not to be prime when $\mathcal{C}$ has more than
one maximal element, which represents a challenge when studying these models.  We get around this issue by using the cumulative distribution function as our new system of coordinates, since marginal independence constraints can be expressed in a simpler way using the cdf.

\begin{defn}[cdf coordinates]\label{def:cdf}
	We define the linear change of coordinates $\varphi: \rr^{\mathcal{R}} \rightarrow \rr^{\mathcal{R}}$, $P\mapsto Q$ by  transforming the tensor $P$ so that the entries of $Q$ are cumulative distributions as follows:
	\begin{equation*}
	\begin{aligned}
q_{j_1\ldots j_n} &= P\left\{X_1\leq j_1,\ldots,X_n\leq j_n\right\}\\
&= \sum_{i_1\leq j_1,\ldots, i_n\leq j_n} p_{i_1\ldots i_n}
	\end{aligned}
	\end{equation*}
\end{defn}
\begin{rmk}
	In the case of binary random variables, cdf coordinates become much simpler. We have that for any $q_{i_1\ldots i_n}$, either $i_l = 1$ or $i_l = 2$, thus we can represent each cdf coordinate by indexing it with its support.
\end{rmk}

We can discuss this transformation, and in particular, its inverse, in terms of the poset of the index vectors of the tensors. Let $\mathcal{P} = [r_1] \times \ldots \times [r_n]$ be the usual direct product poset, where two index vectors $j\leq i$ if $j_k \leq i_k$ for all $k\in [n]$.  Then, we can state the equation for cdf coordinates in Definition \ref{def:cdf} as follows, for all $i\in \mathcal{P}$,

\begin{equation*}
	\begin{aligned}
		q_{i} := \sum_{j\leq i} p_{j}.
	\end{aligned}
	\end{equation*}
Using the M\"obius inversion formula we have that for all $i\in \mathcal{P}$,
\begin{equation*}
	\begin{aligned}
		p_{i} &= \sum_{j\leq i}\mu(j,i) q_{j}\\
		&= \sum_{j\leq i}\prod_{k=1}^n\mu(j_k,i_k) q_{j},
	\end{aligned}
	\end{equation*}
	where 
	\begin{equation*}
		\begin{aligned}
			\mu(j_k,i_k)=\begin{cases}
				1& \mbox{ if } i_k=j_k\\
				-1& \mbox{ if } i_k=j_k+1\\
				0& \mbox{ otherwise.}\\
			\end{cases}
		\end{aligned}
	\end{equation*}
\begin{ex}\label{ex:poset}
	In the case when $r_1=r_2=3$, we have that the poset $\mathcal{P}$ would be as in Figure \ref{fig:poset1}. 

		\begin{figure}
	\centering
		\begin{tikzpicture}[scale=.5]
  \node (23) at (-1,4) {$23$};
  \node (32) at (1,4) {$32$};
  \node (13) at (-2,2) {$13$};
  \node (31) at (2,2) {$31$};
  \node (22) at (0,2) {$22$};
  \node (12) at (-1,0) {$12$};
  \node (21) at (1,0) {$21$};
  \node (11) at (0,-2) {$11$};
  \node (33) at (0,6) {$33$};
  \draw (11) -- (12) -- (13) -- (23) -- (33) -- (32) -- (31) -- (21) -- (11) --(12) -- (22) -- (23) --(22) -- (32) -- (22) -- (21) ;
\end{tikzpicture}
\caption{Index poset $\mathcal P$ from Example \ref{ex:poset}.} \label{fig:poset1}
		\end{figure}

Then, we have that 
	\begin{equation*}
		\begin{aligned}
			 q_{22} = \sum_{(i_1i_2)\leq (2 2)} p_{i_1i_2}.
		\end{aligned}
	\end{equation*}
Hence, using M\"obius inversion on this poset yields 
\begin{equation*}
	\begin{aligned}
		 p_{22} &= \sum_{(i_1i_2)\leq (22)}\mu(i_1i_2,22) q_{i_1i_2}\\
		&= q_{22} - q_{12} - q_{21} + q_{11}.
	\end{aligned}
\end{equation*}
	
\end{ex}
Marginal independence is usually stated in terms of the cumulative probability distribution, which is why this is a natural transformation to consider as it simplifies the equations for marginal independence.

Our goal now is to to explore the ideals of discrete marginal independence models in the 
cdf coordinates.  To do this, we need to explore some properties
of split closed order ideals.
	
	\begin{defn} Let $\mathcal{C}$ be a split closed order ideal of $P\Pi_{n,2}$, then
	\begin{enumerate}
		\item We say that a nonempty set $D\subseteq [n]$ is \emph{disconnected} (with respect to $\mathcal{C}$) if there exists $\pi \in \mathcal{C}$ such that $D=|\pi|$.
		\item A \emph{connected} set is a set that is not disconnected.
		\item A connected set $C$ is \emph{maximal} with respect to some disconnected set $D$ if for any set $C'\subseteq D$ such that $C\subset C'$, then $C'$ is disconnected. 
		\item We define the set $\Con(\mathcal{C}) = \left\{C\subseteq [n]:\; C \mbox{ is connected} \right\}$
	\end{enumerate}
	\end{defn}

	\begin{prop}\label{prop:decomp}
		For any disconnected set $D$ with respect to $\mathcal{C}$, there exists a unique maximal $\pi \in \mathcal{C}$ such that
\[
D=|\pi| = \cup_{i =1}^k\pi_i,
\]
	where all $\pi_i$ are maximal connected sets.
	\end{prop}	
	
	\begin{proof}Follows from Proposition \ref{prop:corralg}.
%
		\end{proof}

Let $i \in \mathcal{R}$.  Define the support of $i$
to be $\supp(i)  =  \{l \in [n] \mid i_l  \neq r_l \}$.  Note that for cdf coordinates 
this will agree
with our previous notion of support, since if $i_l = r_l$ in $q_i$, it means
we are assuming $X_l \leq \infty$ in our cdf calculations, a condition denoted by $+$
in $p_i$.

	\begin{cor}\label{cor:decmp}
	Let $\mathcal{C}$ be a split closed order ideal.
		A probability tensor $P$ lies in the model $\mathcal{M}_\mathcal{C}$ if and only if its cdf coordinates satisfy the following condition.
		 For any $D\subseteq [n]$ that is disconnected by $\mathcal{C}$,  
		 and $i, j^{(l)} \in \mathcal{R}$ be index vectors where
\begin{itemize}
		\item $D= \pi_1\cup\ldots\cup \pi_k$ is the decomposition of $D$ into maximal connected sets.
			\item $D=\supp(i)$, $\pi_l = \supp(j^{(l)})$, and
			\item $i_{\supp(j^{(l)})}=j^{(l)}_{\supp(j^{(l)})}$
			\end{itemize}
			we have
		\begin{equation}\label{eq:decomp}
			q_i = q_{j^{(1)}}\ldots q_{j^{(k)}}.
		\end{equation}
	\end{cor}
	
\begin{proof}
	We know $P$ lies in the model $\mathcal M _{\mathcal{C}}$ if and only if its corresponding cdf satisfies all of the marginal independence constraints. We know that the cdf must satisfy that for all $\pi \in \mathcal{C}$,
	$$q_i = q_{j^{(1)}}\ldots q_{j^{(k)}}$$
	where 
		\begin{itemize}
			\item $|\pi|=\supp(i)$, $\pi_l = \supp(j^{(l)})$, 
			\item $i_{\supp(j^{(l)})}=j^{(l)}_{\supp(j^{(l)})}$.
		\end{itemize}
	If we have a non maximal $\pi$, then we can decompose it further using Proposition $\ref{prop:decomp}$. Hence, the non maximal $\pi$ yield equations that are implied by maximal set partitions in the sense of Proposition \ref{prop:decomp}.
	\end{proof}
	
\begin{rmk}
	Observe that in the case of binary random variables, we get that for any disconnected set $D\subseteq [n]$ with respect to $\mathcal{C}$ with maximal connected sets $\pi_1,\ldots,\pi_k$,
\begin{equation*}
	q_D = q_{\pi_1}\ldots q_{\pi_k}.
\end{equation*}
\end{rmk}

\begin{defn}
	We define the ring of polynomials in cdf coordinates $\rr[Q]$ for the cdf tensor $Q$ of format $\mathcal{R} = [r_1]\times \ldots \times [r_n]$ as	\begin{equation*}
		\rr[Q] = \rr[q_i:\;i\in \mathcal{R}].
	\end{equation*}
\end{defn}

\begin{defn}
	Let $\mathcal{C}$ be a split closed order ideal on $X_1,\ldots, X_n$ and  $i\in \mathcal{R}$ such that $\supp(i)$ is disconnected with respect to $\mathcal{C}$. Then, the associated maximal equation $f_i$ is given by 
	\begin{equation*}
			f_i = q_i - q_{j^{(1)}}\ldots q_{j^{(k)}}
		\end{equation*}
		where 
		\begin{itemize}
			\item $\pi_l = \supp(j^{(l)})$, 
			\item $i_{\supp(j^{(l)})}=j^{(l)}_{\supp(j^{(l)})}$, and 
			\item $\supp(i)= \pi_1\cup\ldots\cup \pi_k$ is the decomposition of $\supp(i)$ into maximal connected sets.
		\end{itemize}

\end{defn}
\begin{defn}
	The \emph{inhomogeneous marginal independence ideal} ($J_{\mathcal{C}}$) associated to the split closed order ideal $\mathcal{C}$ on discrete random variables $X_1,\ldots,X_n$ is
	\begin{equation*}
			J_{\mathcal{C}} = \left\langle f_i \mbox{ maximal}:\;i\in \mathcal{R},\; \supp(i) \mbox{ is disconnected} \right\rangle.
	\end{equation*}
	
\end{defn}

\begin{ex}
	In Example \ref{ex:onest}, there are 16 cdf coordinates indexed by subsets of $[4]$. The equations that define the ideal $J_\mathcal{C}$ without indexing them with their support are
	\begin{equation*}
		\begin{aligned}
			f_{1112} = q_{1112} - q_{1122}q_{2212} = 0,\;\; &
			f_{1212}= q_{1212} - q_{1222}q_{2212} = 0,\;\; &
			f_{2112} = q_{2112} - q_{1212}q_{1122} = 0.
		\end{aligned}
	\end{equation*}
	However, since all random variables are binary, we can instead index them by their support as follows:
	\begin{equation*}
		\begin{aligned}
			f_{123} = q_{123} - q_{12}q_{3} = 0,\;\; &
			f_{13} = q_{13} - q_{1}q_{3} = 0,\;\; &
			f_{23} = q_{23} - q_{2}q_{3} = 0.
		\end{aligned}
	\end{equation*}
	This ideal is prime and therefore toric in cdf coordinates.
\end{ex}
\begin{defn}\label{def:MJC}
	The \emph{marginal independence ideal in cdf coordinates} ($MJ_\mathcal{C}$) associated to the split closed order ideal $\mathcal{C}$ on discrete random variables $X_1,\ldots,X_n$ is
	\begin{equation*}
		\begin{aligned}
			MJ_{\mathcal{C}} = \left\langle 2\times 2 \mbox{ minors of all flattenings of } Q_\pi:\; \pi \in \mathcal{C} \right\rangle,
		\end{aligned}
	\end{equation*}
	where $\pi = \pi_1|\cdots|\pi_k$ and $Q_\pi$ is the $k$-way tensor which is obtained by first marginalizing the tensor $Q$ by making all the indices not in $|\pi|$ equal to their maximum value (this corresponds to summing out indices in Definition \ref{def:MIC}, as we are in cdf coordinates), and then flattening the resulting tensor according to the set partition $\pi$.
	
	Alternatively, $MJ_{\mathcal{C}}$ is the ideal resulting from imposing rank 1 conditions on the tensors $Q_\pi$, for all $\pi \in \mathcal{C}$.
\end{defn}
\begin{ex}
Let us consider the same model as in Example \ref{ex:onest}. Then, $MJ_{12\ind 3}$ is going to be given by imposing rank $1$ conditions on the tensors $Q_{12\ind 3}$, $Q_{1\ind 3}$ and $Q_{2\ind 3}$, all of which are matrices since they are all partial set partitions with 2 blocks. Following the construction in Definition \ref{def:MJC}, we obtain
\begin{equation*}
		Q_{12\ind 3} = \begin{pmatrix}
			q_{1112} & q_{1122}\\
			q_{1212} & q_{1222}\\
			q_{2112} & q_{2122}\\
			q_{2212} & q_{2222}\\
		\end{pmatrix}.
	\end{equation*}
	Furthermore, since the random variables are binary, we can reindex the cdf coordinates by their support as follows:
	\begin{equation*}
		Q_{12\ind 3} =\begin{pmatrix}
			q_{123} & q_{12}\\
			q_{13} & q_{1}\\
			q_{23} & q_{2}\\
			q_{3} & q_{\emptyset}\\
		\end{pmatrix}.
	\end{equation*}
	Observe that all of the rank 1 conditions in $Q_{1\ind 3}$ and $Q_{2\ind 3}$ are found in the $2\times 2$ minors of $Q_{12\ind 3}$. Thus, $MJ_{12\ind 3}$ is given by the $2\times 2$ minors of $Q_{12\ind 3}$
\end{ex}
This ideal and its generators generally have a more complicated structure than $J_\mathcal{C}$. However, once we add a constraint that is imposed by the probability simplex, both ideals have the same generators, which allows us to work with the simpler ideal $J_\mathcal{C}$ to study these models. To prove this result, we first need to prove the following:
\begin{lemma}\label{lem:minor1}
	Let $A = (a_{ij})_{(i,j)\in [m]\times [n]}\in \cc^{m\times n}$ be such that $a_{kl}=1$ for some $(k,l) \in [m]\times [n]$. 
	Then, $\rank(A) = 1$ if and only if the $2\times 2 $ minors involving $a_{kl}$ vanish.
		\end{lemma}
		
	\begin{proof}
		The forward direction is direct. For the converse, without loss of generality (since rank is invariant under row/column operations), assume $a_{11}=1$ and the $2\times 2 $ minors involving $a_{11}$ vanish. Let $a_{i}$ denote the $i$-th row of $A$. To show $A$ is rank 1, it suffices to show $a_{i} $ is a scalar multiple of $a_{1}$, for any $i\in [m]$. We claim that scalar is $a_{i1}$. Let $i\in [m]$, note
		\begin{equation*}
			\begin{aligned}
				a_{i} = a_{i1}a_{1} \iff a_{ij} = a_{i1}a_{1j} \;\; \forall j\in [n],
			\end{aligned}
		\end{equation*}
		which is true as these correspond to the vanishing minors involving $a_{11}$.
	\end{proof}

\begin{prop}
	For a split closed order ideal $\mathcal{C}$, we have that $MJ_\mathcal{C} + \langle q_{r_1\cdots r_n} -1\rangle  = J_\mathcal{C} + \langle q_{r_1\cdots r_n} -1\rangle$.
\end{prop}
\begin{proof}
	First observe that for any $\pi = \pi_1|\cdots |\pi_k \in \mathcal{C}$, there exist partial set partitions $\tau_1,\ldots,\tau_{k} \in \mathcal{C}$ such that
	\[
	\pi = ((\tau_1^{\tau_2})^{\tau_3} \cdots )^{\tau_{k}}
	\]
	these are given by 
	\[
	\tau_i =  \pi_i|\bigcup_{j=i+1}^k\pi_j.
	\]
	In other words, we can get any partial set partition by taking partitions with 2 blocks and applying the splitting operation to them. On the algebraic side, this means the maximal equations $f_i$ that generate $J_\mathcal{C}$ can be obtained if we have all the equations corresponding to set partitions with 2 blocks in $\mathcal{C}$. 
	
	Note that the matrices $Q_\pi$ with $\pi = \pi_1|\pi_2 \in \mathcal{C}$ always have $q_{r_1\cdots r_n}$ appearing. Furthermore, the $2\times 2$ minors that involve $q_{r_1\cdots r_n}$ are of the form 
	\begin{equation*}
		\begin{aligned}
			g_i = q_{r_1\cdots r_n}q_i - q_{j}q_k \in MJ_{\mathcal{C}}
		\end{aligned}
	\end{equation*}
	where $i,j,k \in \mathcal{R}$ are such that
	\begin{itemize}
		\item $\supp(i) = |\pi|$,
		\item $\supp(j) = \pi_1$,
		\item $\supp(k) = \pi_2$,
		\item $i_{\supp(j)} = j_{\supp(j)}$, and
		\item $i_{\supp(k)} = k_{\supp(k)}$
	\end{itemize}
	 Once we add the constraint $q_{r_1\cdots r_n} -1$, we observe that 
	 \begin{equation*}
		\begin{aligned}
			\hat{g_i} = q_i - q_{j}q_k \in MJ_{\mathcal{C}}+ \langle q_{r_1\cdots r_n} -1\rangle. 
		\end{aligned}
	\end{equation*}
	Since any set partition in $\mathcal{C}$ can be obtained by splitting set partitions with 2 blocks, we can keep splitting $q_j$ and $q_k$ until we get the maximal equation $f_i \in J_\mathcal{C}$. Thus showing  $MJ_\mathcal{C} + \langle q_{r_1\cdots r_n} -1\rangle  \supseteq J_\mathcal{C} + \langle q_{r_1\cdots r_n} -1\rangle$.
	
	For the other direction, observe that for any $\pi \in \mathcal{C}$, we know that any flattening of the tensor $Q_\pi$ contains $ q_{r_1\cdots r_n}$. Additionally, under the constraint $q_{r_1\cdots r_n}-1=0$, any such matrix contains a 1. Thus, a flattening of $Q_\pi$ is rank 1 if and only if the minors involving the entry of $1$ vanish, by Lemma \ref{lem:minor1}. However, these $2\times 2$ minors are all equations of the same form as $\hat{g_i}$ above, all of which are in $J_\mathcal{C}$ by construction. Hence showing the tensor $Q_\pi$ is rank 1 if and only if its entries lie in $V(J_\mathcal{C} + \langle q_{r_1\cdots r_n} -1\rangle)$, which in turn shows $J_\mathcal{C} + \langle q_{r_1\cdots r_n} -1\rangle$ generates $MJ_\mathcal{C} + \langle q_{r_1\cdots r_n} -1\rangle$.	
\end{proof}

Now we will exploit the combinatorial structure of the equations in 
$J_\mathcal{C}$ to study these models. 
Observe that Corollary \ref{cor:decmp} gives us a unique factorization for all cdf coordinates and, in particular, it tells us that the cdf coordinates that appear in the defining equations of the variety are precisely those with disconnected support.

\begin{thm}
	For a split closed order ideal $\mathcal{C}\subseteq P\Pi_{n,2}$, the affine  variety $V(J_\mathcal{C})$ is toric in cdf coordinates.
\end{thm}	
	
	\begin{proof}
		By Corollary \ref{cor:decmp}, we know that the ideal $J_\mathcal{C}$ that generates the model is given by equations of the form
		\begin{equation*}\label{eq:decomp}
			f_i = q_i - q_{j^{(1)}}\ldots q_{j^{(k)}}
		\end{equation*}
		where 
		\begin{itemize}
			\item $D$ is disconnected. 
			\item $D=\supp(i)$, $\pi_l = \supp(j^{(l)})$, 
			\item $i_{\supp(j^{(l)})}=j^{(l)}_{\supp(j^{(l)})}$ and 
			\item $D= \pi_1\cup\ldots\cup \pi_k$ is the decomposition of $D$ into maximal connected sets.
		\end{itemize}
		Now, let us consider the term order $\prec$ on $\rr[Q]$ defined by a lexicographic order with the condition that if $|\supp(i)| < | \supp(j)|$, then $q_i\prec q_j$. Therefore, for each generator $f_i$ of $J_\mathcal{C}$, we have that $\mbox{lt}_\prec(f_i) = q_i$. 
		
		This is a Gr\"obner basis for $J_\mathcal{C}$ since the leading terms of any pair of equations are coprime, so any $S$-pair reduces to $0$ modulo the basis of difference of binomials.  Thus, the initial ideal of $J_\mathcal{C}$, ${\rm in}_\prec(J_{\mathcal{C}})$ is generated by linear terms, so it is prime. This implies that $J_\mathcal{C}$ is a prime ideal generated by differences of binomials, hence it is toric. 
	\end{proof}

The explicit description of the generators of $J_\mathcal{C}$ also allows us to show that its associated affine variety is smooth. 
\begin{ex}
Consider the split closed order ideal $\mathcal{C}$ on 3 binary random variables generated by the independence statement $X_1\ind X_{2,3}$. Then, by Corollary \ref{cor:decmp}, 
$$J_{\mathcal C} = \left\langle q_{12}-q_1q_2,q_{13}-q_1q_3, q_{123}-q_1q_{23} \right\rangle$$
Hence, the Jacobian of $J_\mathcal{C}$ is 
\begin{equation*}
	\mbox{Jac}(J_\mathcal{C}) = \begin{bmatrix}
		0&-q_2 & -q_1 & 0 & 1 & 0 & 0 & 0\\
		0&-q_3 & 0 & -q_1 & 0 & 1 & 0 & 0\\
		0&-q_{23} & 0 & 0 & 0 & 0 & -q_1 & 1\\
	\end{bmatrix}.
\end{equation*}
Observe that $J_\mathcal{C}$ is smooth everywhere since its Jacobian has full rank independent of the point of evaluation. 
\end{ex}

\begin{thm}
For a split closed order ideal $\mathcal{C}\subseteq P\Pi_{n,2}$, the affine variety $V(J_{\mathcal{C}})$ is smooth in cdf coordinates.
\end{thm}
\begin{proof}
	It suffices to show that the Jacobian of $J_\mathcal{C}$ is full rank. Observe that $J_\mathcal{C}$ is generated by equations of the form 
		\begin{equation}
			f_i = q_i - q_{j^{(1)}}\ldots q_{j^{(k)}}
		\end{equation}
		where 
		\begin{itemize}
			\item $D=\supp(i)$, $\pi_l = \supp(j^{(l)})$, 
			\item $i_{\supp(j^{(l)})}=j^{(l)}_{\supp(j^{(l)})}$ and 
			\item $D= \pi_1\cup\ldots\cup \pi_k$ is the decomposition of $D$ into maximal connected sets.
		\end{itemize}
		for all disconnected $D$. Hence, the equations that appear on the generating set are precisely those indexed by $i$ such that $\supp(i)$ is disconnected. In addition,
		observe that we have that for all index vectors $i$ and $j$ with disconnected support,
		$$\frac{\partial}{\partial q_j}f_i=\begin{cases}
			1 & \mbox{if $j=i$}\\
			0 & \mbox{else}
		\end{cases}.$$
		This follows because of the uniqueness of the decomposition of $q_i$. Hence, the columns of the Jacobian contain some permutation of the columns of the identity matrix of size $\#\{i\in [r_1]\times \cdots \times [r_n]: \supp(i) \mbox{ is disconnected}\}$, which shows that $\mbox{Jac}(J_{\mathcal{C}})$ has full rank everywhere. 
\end{proof}

We are now ready to prove a result about existence of distributions in the binary case.
\begin{prop}\label{prop:modexist}
For each split closed order ideal $\mathcal{C}\subseteq P\Pi_{n,2}$, there exists a probability distribution $P$ in the binary model $\mathcal{M}_{\mathcal{C}}$ such that $P \in \Delta_{2^n}\setminus\partial(\Delta_{2^n})$ and $P$ satisfies all the marginal independence statements in $\mathcal{C}$ and no other marginal independence statements.
\end{prop}
\begin{proof}
We want to show there is a point in the interior of the probability simplex that satisfies exactly the statements in $\mathcal{C}$ and no others. It suffices to show this for cdf coordinates, as the interior of the probability simplex in probability coordinates maps to the interior of the probability simplex in cdf coordinates.

Since we are using binary random variables, the description of our generators becomes much simpler, $$J_{\mathcal{C}}=\langle f_D |\; D \mbox{ disconnected} \rangle.$$
Take $\pi \notin \mathcal{C}$ to be minimal, this corresponds to an equation $g = q_\pi - q_{\pi_1}\cdots q_{\pi_k}$ where $\pi$ is connected. Let 
\[
I = J_\mathcal{C} + \langle q_\pi - q_{\pi_1}\cdots q_{\pi_k} \rangle
\]
Observe that ${\rm in}_\prec(I) = {\rm in}_\prec(J_\mathcal{C}) + \langle q_\pi \rangle$. 
So that we still get a Gr\"obner basis for $I$ composed of linear terms and therefore, $\dim(I)<\dim(J_\mathcal{C})$. Hence, we have a strict inclusion, $V(I)\subset V(J_\mathcal{C})$, showing that $V(J_\mathcal{C})\setminus V(I) \neq \emptyset$. 
Since the uniform distribution $U  \in V(J_\mathcal{C})$, it intersects the interior of the probability simplex (in cdf coordinates). Finally, since we additionally know $V(J_\mathcal{C})$ is smooth, we conclude that $\dim(V(J_\mathcal{C})) = \dim(V(J_\mathcal{C})\cap \Delta_{2^n})) $.
\end{proof}

With Proposition \ref{prop:modexist} in hand,  we are ready to prove Theorem \ref{thm:bijection}.
\begin{proof}[Proof of Theorem \ref{thm:bijection}]\

\noindent $\left(\implies\right)$
	Given a model $\mathcal{M}_\mathcal{C}$, we obtain a unique split closed order ideal given by taking the split closure of $\mathcal{C}$. We know any probability distribution satisfying $\mathcal{C}$ satisfies its closure by Proposition \ref{prop:iffsplit}.

\noindent $\left(\impliedby\right)$	
	Now let $\mathcal{C} \in P\Pi_{n,2}$ be a split closed order ideal. It suffices to show that for any $\pi\notin \mathcal{C}$, there exists a distribution $P$ that satisfies the statements in $\mathcal{C}$ and does not satisfy $\pi$, this follows directly from Proposition \ref{prop:modexist}.
\end{proof}

We have now proven the bijective correspondence between split closed order ideals and marginal independence models, which provides us with a combinatorial description of this class of models.

So far, we have been describing our models and varieties in affine space. However, Proposition \ref{prop:decomp} and Corollary \ref{cor:decmp} allow us to give an explicit homogeneous parametrization of the ideals defining our models.

\begin{cor}\label{prop:par}
	Let $\rr [\Theta]:= \rr \left[t, \theta_l^{(C)}:\; C\in \Con(\mathcal{C}),\; \supp(l) = C,\; l\in \prod_{i=1}^n[r_i]\right]$. A homogeneous parametrization of the model in cdf coordinates is given by the homomorphism between polynomial rings $\phi_\mathcal{C}: \rr [Q] \rightarrow \rr[\Theta]$ defined by
	\begin{equation*}
		\phi_\mathcal{C}(q_{i_1\ldots i_n}) = t \prod_{j=1}^k \theta_{i_{\pi_j}}^{(\pi_j)},
	\end{equation*}
	where $\supp(i) = \pi_1\cup \ldots \cup \pi_k$ is the decomposition of $\supp(i)$ into maximal connected sets.
	\end{cor}
\begin{rmk}
In the binary case, we can lose the subscript on our parameters $\theta$ as a consequence of the reindexing of binary cdf coordinates with their support.
	
\end{rmk}

\begin{defn}
	The \emph{homogeneous marginal independence ideal} ($I_{\mathcal{C}}$) associated to the split closed order ideal $\mathcal{C}$ on discrete random variables $X_1,\ldots,X_n$ is given by 
	\begin{equation*}
			I_{\mathcal{C}} = \ker \phi_\mathcal{C},
	\end{equation*}
	where $\phi_\mathcal{C}$ is the parametrization given in Corollary $\ref{prop:par}$.
\end{defn}

\begin{prop}
	For a split order model ideal $\mathcal{C}$, we have
	$$I_\mathcal{C} = \overline{J_{\mathcal{C}}},$$
	where $\overline{J_{\mathcal{C}}}$ is the homogenization of $J_\mathcal{C}$. 
\end{prop}
\begin{proof}\
This follows from construction of the parametrization in Corollary \ref{prop:par} and the definition of $I_\mathcal{C}$.
\end{proof}

In general, toric varieties are defined by a monomial map with an associated integer matrix $A_\mathcal{C}$. This parametrization allows us to compute this integer matrix and the associated polytope $\mathfrak{P}_\mathcal{C}$ given by the convex hull of the columns of $A_\mathcal{C}$.
\begin{ex}\label{ex:2st}
	Consider the  model on 4 binary variables given by $X_{1} \ind X_{2} \ind X_{3,4}$. Then, we have 16 cdf coordinates, and the equations that define $J_\mathcal{C}$ are:
	\begin{equation*}
		\begin{aligned}
		q_{12}-q_{1}q_{2} &, &
		q_{13}-q_{1}q_{3} &, &
		q_{14}-q_{1}q_{4} &, &
		q_{23}-q_{2}q_{3} &, &
		q_{24}-q_{2}q_{4} &, \\
		q_{234}-q_{2}q_{34} &, &
		 q_{123} - q_{1}q_{2}q_{3} &,&
		q_{124} - q_{1}q_{2}q_{4} &, &
		q_{1234} - q_{1}q_{2}q_{34}
	\end{aligned}
	\end{equation*}

From this description, we get the generators for $J_{\mathcal{C}}$ and consequently all coordinates with disconnected support. Hence, the parametrization described in Corollary \ref{prop:par} yields the following integer matrix:
	\begin{equation*}
		A_\mathcal{C} = 
    \kbordermatrix{ & q & q_1 & q_2& q_3& q_4 & q_{12} & q_{13} & q_{14} & q_{23}& q_{24} &q_{34} & q_{123} & q_{124} & q_{134}& q_{234} & q_{1234}      \\
      				 t  & 1 & 1 & 1 & 1 & 1 & 1 & 1 & 1 & 1 & 1 & 1 & 1 & 1 & 1 & 1 & 1\\
      \theta^{(1)}  & 0 & 1 & 0 & 0 & 0 & 1 & 1 & 1 & 0 & 0 & 0 & 1 & 1 & 1 & 0 & 1 \\
      \theta^{(2)}  & 0 & 0 & 1 & 0 & 0 & 1 & 0 & 0 & 1 & 1 & 0 & 1 & 1 & 0 & 1 & 1 \\
      \theta^{(3)}  & 0 & 0 & 0 & 1 & 0 & 0 & 1 & 0 & 1 & 0 & 0 & 1 & 0 & 0 & 0 & 0 \\
      \theta^{(4)}  & 0 & 0 & 0 & 0 & 1 & 0 & 0 & 1 & 0 & 1 & 0 & 0 & 1 & 0 & 0 & 0 \\
      \theta^{(34)} & 0 & 0 & 0 & 0 & 0 & 0 & 0 & 0 & 0 & 0 & 1 & 0 & 0 & 1 & 1 & 1 }
	\end{equation*}
	Furthermore, we know that the ideal $I_\mathcal{C}$ is going to be the kernel of the monomial map given by $A_\mathcal{C}$, which, in this case we compute using Macaulay2 \cite{M2} and find that there are 46 degree 2 generators of the ideal. They include homogenized equations corresponding to the generators of $J_\mathcal{C}$
	\begin{equation*}
		\begin{matrix}
			q_{1}q_{2}-qq_{12} \;& q_{1}q_{3}-qq_{13}\;& q_{1}q_{4}-qq_{14} \;&q_{2}q_{3}-qq_{23} \;&  q_{2}q_{4}-qq_{24}\;& q_{1}q_{34}-qq_{134}\\
			q_{2}q_{34}-qq_{234}\;&q_{1}q_{23}-qq_{123}\;&q_{2}q_{13}-qq_{123}  \;&q_{3}q_{12}-qq_{123}\;&q_{1}q_{24}-qq_{124}\; &q_{2}q_{14}-qq_{124}\\
     		q_{4}q_{12}-qq_{124}&q_{1}q_{234}-qq_{1234}\;& q_{2}q_{134}-qq_{1234}\;& q_{12}q_{34}-qq_{1234}.
		\end{matrix}
	\end{equation*}
	However, we find more equations in the minimal generators of the homogeneous ideal, such as
	\begin{equation*}
	\begin{matrix}
			&q_{134}q_{234}-q_{34}q_{1234}\;& q_{124}q_{234}-q_{24}q_{1234}\;& q_{123}q_{234}-q_{23}q_{1234}\;& q_{14}q_{234}-q_{4}q_{1234}\\
		\end{matrix}
	\end{equation*}
	which are not necessary in the de-homogenized cdf coordinates since they are implied by the original equations (when $q=1$).	
\end{ex}

\section{Graphical and simplicial marginal independence models}\label{sec:graphsimp}
Two particularly interesting classes of marginal independence models are discussed in \cite{boege} and in \cite{drton}. They are graphical models and what we will refer to as simplicial models. We aim to contrast these different classes of models and show how they fit in the framework presented above.
\subsection{Graphical models}
Graphical models have been studied extensively because of their many applications and interesting structure that comes with a graphical representation of the  independence relations. Different types of graphs are used to express different relationships between variables: directed acyclic graphs are been used to study causality in Bayesian networks, undirected graphs on the other hand, are used to study conditional independence models. Following the work of Drton and Richardson in \cite{drton}, we use bidirected graphs as a framework to discuss marginal independence. Hence, for this section, $G = (V,E)$ will be a bidirected graph with vertex set $V = [n]$  and a set of edges $E$.
\begin{defn}
	Let $G = (V,E)$ be a bidirected graph, for a vertex $v\in V$ we define the \textit{spouse} of $v$ as the set of all vertices that share an edge with $v$ and $v$ itself, i.e., $\spou(v):=\{w\in V:\; (w,v) \in E\} \cup \{v\}$. For a subset $A\subseteq V$, we define $\spou(A) = \cup_{v\in A} \spou(v)$.
\end{defn}
There are several different ways in which we can read different kinds of independent statements from a graph. The rules for reading a graph are called \textit{Markov properties}, we are particularly interested in two sets of Markov properties: the global Markov property and the connected set Markov property.

\begin{defn}
	Let $G = \left([n],E\right)$ be a bidirected graph. A probability distribution is said to satisfy the \emph{global Markov property} if for all $A,B,C \subseteq [n]$ pairwise disjoint such that $[n]\setminus A\cup B \cup C$ separates $A$ and $B$, we have that
	$$X_{A}\ind X_{B} |X_{C}.$$
\end{defn}

\begin{defn}
	A probability distribution $P$ is said to satisfy the \emph{connected set Markov property} associated to a bidirected graph $G$ if $$X_C \ind X_{V\setminus \spou(C)}$$
	where $C$ is a connected subgraph of $G$.
\end{defn}

	The global Markov property contains all of the marginal statements given by the connected set Markov property since for any connected set $C$, we have that $C$ is separated from $V\setminus \spou(C)$ by $\spou(C)\setminus C = [n]\setminus(C \cup [n]\setminus\spou(C))$ (in this case, the third set would be empty). Even though the global Markov property has conditional statements that are not given by the connected set property, we know that a probability distribution $P$ satisfies the global Markov property if and only if it satisfies the connected set property. 
	
		Additionally, since $v\in V$ is connected, the connected set Markov rules imply that if $u,v\in V$ do not have an edge between them, then $X_v \ind X_u$. This is known as the \emph{pairwise Markov property}, hence, both the global and connected set Markov properties imply all the statements in the pairwise Markov property.
		
\begin{ex}
Consider the chain with 4 vertices.

The connected set Markov property yields $X_{1}\ind X_{3,4}$ and $X_{4}\ind X_{1,2}$ and the global Markov property gives us $X_1\ind X_4 |X_{2,3}$ as well.
\end{ex}

\begin{defn}[Graphical model]
	The marginal independence graphical model associated to $G$, $\mathcal{M}_G$, is given by the set of distributions $P = (p_{i_1,\ldots,i_n})$ and discrete random variables $X_1,\ldots,X_n$ that satisfy the marginal independence statements given by the connected set Markov property of $G$.
\end{defn}
\begin{ex}
	Let $G$ be the graph on 3 vertices with the edge $(1,2)$ and binary random variables.

Then, the marginal independence statements in this model are given by $X_{1,2} \ind X_3$ and the equations defining $J_\mathcal{C}$ are
\begin{equation*}
	 	\begin{aligned}
	 		q_{123}-q_{12}q_3,\;\;&
	 		q_{13}-q_1q_3,\;\; &
	 		q_{23}-q_2q_3.
	 	\end{aligned}
	 \end{equation*}
	 These are the same defining equations for the model as the ones in Example \ref{ex:onest}. However, they represent slightly different models since here we have 3 random variables instead of 4.
\end{ex}
In general, given a graph $G$, the collection $\mathcal{C}$ of the marginal independence statements is defined by all the connected set Markov statements associated to the graph. This means that we can apply the results from Section \ref{sec:discretemods} to this special class of models. Furthermore, the notion of connectedness relative to a collection of statements $\mathcal{C}$ agrees with the notion of connectedness of subgraphs of the corresponding graph $G$.

\begin{prop}\label{prop:connec}
	Let $G$ be a bidirected graph and $\mathcal{C}$ the split closed order ideal defined by the connected set Markov statements. Then a set $D \subseteq [n]$ is disconnected with respect to $\mathcal{C}$ if and only if the corresponding subgraph with vertices in $D$ is disconnected.
\end{prop}
	
	\begin{proof}
		If $D$ is disconnected with respect to $\mathcal{C}$, then there exists $\pi \in \mathcal{C}$ such that $D=|\pi|$. Without loss of generality, we may assume that $\pi_1$ is connected with respect to $\mathcal{C}$ and that it only has two blocks, so, $\pi_2 = V\setminus\spou(\pi_1)$. Thus, the subgraph induced by $D$ is disconnected.
		
		Conversely, if the subgraph of $D$ is disconnected, then we can certainly find a connected component $\pi$ of $D$, and take $\pi | (V\setminus\spou(\pi))$ so that $D$ is disconnected with respect to our collection $\mathcal{C}$.
	\end{proof}

Proposition \ref{prop:connec} ensures that in this particular class of models, we can restate the results of Section \ref{sec:discretemods} in terms of the graph theoretic sense of connectedness, this agrees with the results presented in \cite{drton} in the binary case and allows us to extend them to the general case.

\begin{ex}\label{ex:3cyc}
	Let $\mathcal{M}_{\mathcal{C}}$ be the model on 3 binary variables given by the marginal independence statements $X_1\ind X_2$, $X_2 \ind X_3$, $X_1\ind X_3$. The only graphical model on 3 binary variables that contains these marginal independence statements is given by the graph on 3 vertices that has no edges. However, because of the connected set Markov rule, this graphical model also has the statement $X_1\ind X_2 \ind X_3$, which cannot be inferred from $\mathcal{C}$. We will show in the following section that this model is a simplicial complex model.
\end{ex}
\subsection{Simplicial complex models}
This class of models was proposed in \cite{boege} and it concerns a very different class of models. We obtain these models from a simplicial complex instead of a graph, and the way we read statements from simplices is that all of the random variables indexed by a face are completely independent. Furthermore, this can be achieved by imposing rank 1 constraints on the probability tensor indexed by any particular face, which consequently makes the factorization in cdf coordinates and the homogeneous parametrization much simpler. 
\begin{defn}[Simplicial complex model]
	Let $\Sigma \subseteq 2^{[n]}$ be a simplicial complex, then the model $\mathcal{M}_{\Sigma}$ is defined by the distributions $P = (p_{i_1\ldots i_n})$ and finite random variables such that the random variables $\{X_i: i\in \sigma\}$ are completely independent for any face $\sigma \in \Sigma$.
\end{defn}
\begin{ex}
Let $\Sigma$ be the simplicial complex on 3 binary variables defined by the 3 cycle, i.e., the marginal independence relations would be exactly the same as in Example \ref{ex:3cyc}. This shows that there are simplicial models that are not graphical models.
\end{ex}

\begin{rmk}
	We can discuss this class of models in terms of general marginal independence models by defining a collection of marginal independence statements generated by $\mathcal{C} = \left\{\sigma_1|\cdots| \sigma_k: \sigma \in \Sigma \right\}$. Thus, we can apply all of the results from Section \ref{sec:discretemods} to simplicial complex models, i.e. we get a toric model in cdf coordinates and a parametrization in terms of the connected sets. By construction, the disconnected sets of $\mathcal{C}$ are the faces of $\Sigma$ (except the vertices), and the connected sets are the non-faces and the vertices.
\end{rmk}
Next, we wish to compare graphical models and simplicial complex models. We need to be careful when comparing the two classes of models since the way we read the marginal independence statements from their corresponding structures is very different: in a simplicial complex, any face gives an marginal independence statement (in particular, edges represent marginal independence relations), whereas in graphs, any non-edge gives a marginal independence statement (and we get more from Markov rules). For this reason, when we are comparing them, it is useful to talk about the complementary graph $G'$ of $G$. 

In general, we can associate a simplicial complex model to any graphical model such that the graphical model is a submodel of the simplicial complex model.
\begin{prop}\label{prop:gsimp}
	Let $G = (V,E)$ be a graph and $G'$ be its complementary graph, $G' = (V, E')$ where $E'$ is the usual complement of $E$. Then, the associated simplicial complex, $\Sigma$, given by
	\begin{equation*}
		\Sigma(G) = \left\{F\subseteq V \; \big|\; e\nsubseteq F \; \forall e \in E\right\}
	\end{equation*}
	gives us that $\mathcal{M}_G \subseteq \mathcal{M}_{\Sigma}$. 
\end{prop}
	\begin{proof}
			We need to show the containment of the ideals $I_{\Sigma(G)} \subseteq I_G$, which in turn requires us to show that any marginal independence statement from $\Sigma(G)$ can be obtained from $G$ as well. The edges (1 dimensional faces) of $\Sigma(G)$ are exactly the non-edges in $G$ and the faces of $\Sigma(G)$ are the completely disconnected subgraphs of $G$. Hence, every marginal independence statement from $\Sigma(G)$ can be found in the model associated to $G$. Thus proving that $\cm_G$ is a submodel of $\cm_{\Sigma(G)}$.
	\end{proof}
\begin{rmk}
Observe that the construction in Proposition \ref{prop:gsimp} for a given $G$ implies that the $1$-skeleton of $\Sigma(G)$ is the complimentary graph $G'$. The existence of a simplicial complex model that  contains a given graphical model raises the question of when the simplicial complex models and graphical models are isomorphic. Consider the following example:	
\end{rmk}

	\begin{ex}\label{ex:4cyc}
		Let $\mathcal{M}_G$ be the model associated to the binary 4-cycle.
	\begin{figure}
	\centering
	\subfigure[4-cycle graph $G$]{
		\begin{tikzpicture}[scale=1]
	\tikzset{myarrow/.style={postaction={decorate},
decoration={markings,
mark=at position 1 with {\arrow{latex}},
}}}
  \node (1) at (0,1.6) {$1$};
  \node (2) at (1.6,0) {$2$};
  \node (3) at (0,-1.6) {$3$};
  \node (4) at (-1.6,0) {$4$};
  \draw[myarrow] (1) -- (2) ;
  \draw[myarrow] (2) -- (1) ;
  \draw[myarrow] (2) -- (3) ;
  \draw[myarrow] (3) -- (2) ;
  \draw[myarrow] (3) -- (4) ;
  \draw[myarrow] (4) -- (3) ;
  \draw[myarrow] (1) -- (4) ;
  \draw[myarrow] (4) -- (1) ;
\end{tikzpicture}}
\qquad
\subfigure[Simplicial complex $\Sigma_G$]{
		\begin{tikzpicture}[scale=1]
	\tikzset{myarrow/.style={postaction={decorate},
decoration={markings,
mark=at position 1 with {\arrow{latex}},
}}}
  \node (1) at (0,1.6) {$1$};
  \node (2) at (1.6,0) {$2$};
  \node (4) at (0,-1.6) {$4$};
  \node (3) at (-1.6,0) {$3$};
  \draw (1) -- (3) ;
  \draw (2) -- (4) ;
\end{tikzpicture}}
\caption{Graph and simplicial complex in Example \ref{ex:4cyc}} \label{fig:graphsimp}
		\end{figure}
This gives us the marginal independence relations $X_1 \ind X_3$ and $X_2\ind X_4$. As we can see in Figure \ref{fig:graphsimp}, the simplicial complex $\Sigma_G$ associated to this graph gives us the same marginal independence statements as $G$, i.e. $\mathcal{M}_G=\mathcal{M}_{\Sigma(G)}$.
	\end{ex}

	\begin{ex}\label{ex:4chain}
		Let $\mathcal{M}_G$ be the model associated to the binary 4-chain.
	\begin{figure}
	\centering
	\subfigure[4-chain graph $G$]{
		\begin{tikzpicture}[scale=1]
	\tikzset{myarrow/.style={postaction={decorate},
decoration={markings,
mark=at position 1 with {\arrow{latex}},
}}}
  \node (1) at (-3,0) {$1$};
  \node (2) at (-1,0) {$2$};
  \node (3) at (1,0) {$3$};
  \node (4) at (3,0) {$4$};
  \draw[myarrow] (1) -- (2) ;
  \draw[myarrow] (2) -- (1) ;
  \draw[myarrow] (2) -- (3) ;
  \draw[myarrow] (3) -- (2) ;
  \draw[myarrow] (3) -- (4) ;
  \draw[myarrow] (4) -- (3) ;
\end{tikzpicture}}
\qquad
\subfigure[Simplicial complex $\Sigma_G$]{
		\begin{tikzpicture}[scale=1]
	\tikzset{myarrow/.style={postaction={decorate},
decoration={markings,
mark=at position 1 with {\arrow{latex}},
}}}
  \node (1) at (-3,0) {$3$};
  \node (2) at (-1,0) {$1$};
  \node (3) at (1,0) {$4$};
  \node (4) at (3,0) {$2$};
  \draw (1) -- (2) ;
  \draw (2) -- (3) ;
  \draw (3) -- (4) ;
\end{tikzpicture}}
\caption{Graph and simplicial complex in Example \ref{ex:4chain}.} \label{fig:graphsimpchain}
		\end{figure}
This graph gives us the following marginal independence statements: $X_1 \ind X_{3.4}$, $X_4 \ind X_{1,2}$. Its associated simplicial complex, $\Sigma_G$, is again a 4 chain that implies the following: $X_1 \ind X_3$, $X_1 \ind X_4$ and $X_2 \ind X_4$. We cannot derive the marginal independence statements given by $G$ from this	. Thus, in this case the models are not isomorphic (even though the 1-skeleton of $\Sigma(G)$ and $G$ are isomorphic as graphs). 
	\end{ex}
For a graphical model $\mathcal{M}_G$ to be isomorphic to its simplicial complex counterpart, we would need that the marginal independence statements on $G$ are only statements of complete marginal independence of individual variables, since those are the only type of marginal independence statements that we can obtain from the simplicial complex model.
\begin{prop}\label{prop:isom}
	For a graph $G$, $\mathcal{M}_G = \mathcal{M}_{\Sigma(G)}$ if the complimentary graph $G'$ of $G$ is a disjoint union of complete graphs.
\end{prop}
\begin{proof}
		Let $G'=G_1'\cup \ldots \cup G_k'$ where each $G_i'$, $i\in [k]$, is a complete graph and they are disjoint. Then, observe that each $G_i'$ corresponds to a totally disconnected subgraph of $G$ whose vertices are connected all vertices not in $G_i'$. Hence, each $G_i'$ yields a complete marginal independence statement on the random variables indexed by the vertices of $G_i'$. Furthermore, by construction, these are all the statements that we are going to get from the graph $G$. We know that those are precisely the statements that we get from $\Sigma(G)$, thus showing that the models are isomorphic.
	\end{proof}
\begin{rmk}
We can restate Proposition \ref{prop:isom} without using the complimentary graph since that condition on $G'$ is equivalent to $G$ being a multipartite graph. Additionally, we can see from this proposition that the number of nontrivial isomorphic models is going to be the number of partitions of $\#V$.
\end{rmk}
\begin{ex}
	Let $X_1,\ldots,X_4$ be finite random variables, and consider the model given by the marginal independence statements $X_{1}\ind X_{2,3,4}$ and $X_2 \ind X_3$. This model is not simplicial and not graphical, but it still is toric.
\end{ex}

\subsection{Counting models and computational results.}\

We counted all models on 3 and 4 random variables and counted up to symmetry of relabeling the random variables and obtained the following:
\begin{table}[h]
	\caption{Number of models on $3$ and $4$ binary random variables} \label{tab:count}
	\begin{tabular}{|c|c|c|c|c|}
	\hline
	& Total& Up to symmetry &Total & Up to symmetry \\
	&$n=3$&$n=3$&$n=4$&$n=4$\\
	\hline
	Graphical& 8&4 &64 &11  \\
	\hline
	Simplicial& 9&5 &114 &20 \\
	\hline
	Graphical and simplicial& 5&3 &18 &5 \\
	\hline
	General&12&6 &496 &53\\
	\hline
\end{tabular}

\end{table}

For the simplicial and graphical models, we used data from OEIS \cite{unlabeledGr,unlabeledSC,labeledSC}. For the general case, we do not yet have an efficient way of counting split closed order ideals, so we used exhaustive computation checking split closures of order ideals.

Additionally, we computed the homogeneous ideals for all models up to labeling symmetry for  $n=4$ and found the highest degree of the generators, degree and dimension of the projective variety that comes from $I_\mathcal{C}$ and whether or not the models are simplicial or graphical.
	
	\begin{longtable}{|l|l|l|l|c|c|}
\caption{Computations for all models on 4 random variables.} \label{tab:long} \\

\hline 
\multicolumn{1}{|c|}{\textbf{Split order}} & \multicolumn{1}{c|}{\textbf{Max degree}} & \multicolumn{1}{c|}{\textbf{Degree}}& \multicolumn{1}{c|}{\textbf{Dimension}}& \multicolumn{1}{c|}{\textbf{Graphical}}& \multicolumn{1}{c|}{\textbf{Simplicial}} \\
 \multicolumn{1}{|c|}{\textbf{ideal}} & \multicolumn{1}{c|}{\textbf{gens}} & \multicolumn{1}{c|}{}& \multicolumn{1}{c|}{}& \multicolumn{1}{c|}{}& \multicolumn{1}{c|}{} \\
 \hline 
\endfirsthead

\multicolumn{6}{c}%
{{\bfseries \tablename\ \thetable{} -- continued from previous page}} \\
\hline 
\multicolumn{1}{|c|}{\textbf{Split order}} & \multicolumn{1}{c|}{\textbf{Max degree}} & \multicolumn{1}{c|}{\textbf{Degree}}& \multicolumn{1}{c|}{\textbf{Dimension}}& \multicolumn{1}{c|}{\textbf{Graphical}}& \multicolumn{1}{c|}{\textbf{Simplicial}}  \\
 \multicolumn{1}{|c|}{\textbf{ideal}} & \multicolumn{1}{c|}{\textbf{gens}} & \multicolumn{1}{c|}{}& \multicolumn{1}{c|}{}& \multicolumn{1}{c|}{}& \multicolumn{1}{c|}{} \\
\hline 
\endhead

\hline \multicolumn{6}{|r|}{{Continued on next page}} \\ \hline
\endfoot

\hline
\endlastfoot
$\emptyset$&&1&15&x&x\\
$1 | 2 $&2&2&14&x&x\\
$1 | 2,  1 | 3 $&2&3&13& & x\\
$1 | 2,  3 | 4 $&2&4&13& x & x\\
$1 | 2 3 $&2&4&12& x & \\
$1 | 2,  1 | 3,  1 | 4 $&2&4&12& & x\\
$1 | 2,  1 | 3,  2 | 3 $&3&5&12&x & x\\
$1 | 2,  1 | 3,  2 | 4 $&2&5&12& & x\\
$1 | 2,  1 | 3,  1 | 4,  2 | 3 $&3&7&11& & x\\
$1 | 2,  1 | 3,  2 | 4,  3 | 4 $&2&6&11& & x\\
$1 | 2,  1 | 3 4 $&2&5&11& & \\
$1 | 2,  1 3 | 4 $&2&7&11& & \\
$1 | 2 | 3 $&2&6&11& & x\\

$1 | 2,  1 | 3,  1 | 4,  2 | 3,  2 | 4 $&3&9&10& & x\\

$1 | 2,  1 | 3,  2 | 3 4 $&3&9&10& & \\
$1 | 2,  1 | 3,  2 3 | 4 $&2&9&10& & \\
$1 | 2,  1 | 3 | 4 $&2&8&10& & x\\
$1 | 2 3,  1 | 2 4 $&2&6&10& & \\
$1 | 2 3,  1 4 | 2 $&2&10&10&x & \\

$1 | 2,  1 | 3,  1 | 4,  2 | 3,  2 | 4,  3 | 4 $&3&12&9& & x\\
$1 | 2,  1 | 3,  1 | 4,  2 | 3 4 $&3&12&9& & \\
$1 | 2,  1 | 3,  2 | 3 | 4 $&3&11&9& & x\\
$1 | 2,  1 3 | 4,  2 3 | 4 $&3&11&9& & \\
$1 | 2,  1 3 | 4,  2 4 | 3 $&2&14&9& & \\
$1 | 2 3,  1 | 2 4,  1 | 3 4 $&2&7&9& & \\

$1 | 2 3,  2 3 | 4 $&2&10&9& & \\
$1 | 2 | 3,  1 | 2 4 $&2&10&9& & \\

$1 | 2,  1 | 3,  1 | 4,  2 | 3 | 4 $&3&16&8& & x\\
$1 | 2,  1 | 3,  1 2 3 | 4 $&3&15&8& & \\
$1 | 2,  1 | 3,  1 2 3 | 4,  2 | 3 $&3&17&8& & \\
$1 | 2,  1 2 3 | 4 $&3&13&8& x& \\
$1 | 2,  1 | 3 | 4,  2 3 | 4 $&3&14&8& & \\
$1 | 2,  1 | 3 4,  2 | 3 4 $&3&14&8& & \\

$1 | 2 3,  1 4 | 2,  1 4 | 3 $&2&16&8& & \\
$1 | 2 3 4 $&2&8&8&x & \\

$1 | 2 | 3,  1 | 2 4,  1 | 3 4 $&2&12&8& & \\

$1 | 2 | 3,  1 | 2 | 4 $&2&12&8& & x\\

$1 | 2,  1 | 3,  1 2 3 | 4,  2 | 3 | 4 $&3&21&7& & \\
$1 | 2,  1 | 3 | 4,  2 | 3 | 4 $&3&18&7& & x\\
$1 | 2,  1 | 3 | 4,  1 2 3 | 4 $&3&17&7& & \\
$1 | 2 3,  1 4 | 2,  1 4 | 3,  2 3 | 4 $&2&19&7& & \\

$1 | 2 | 3,  1 | 2 3 4 $&2&14&7& & \\

$1 | 2 | 3,  1 | 2 4,  2 4 | 3 $&3&17&7& & \\
$1 | 2 | 3,  1 | 2 | 4,  1 | 3 4 $&2&15&7& & \\

$1 | 2,  1 | 3 | 4,  1 2 3 | 4,  2 | 3 | 4 $&3&23&6& & \\
$1 2 | 3 4 $&2&20&6& x & \\
$1 | 2 | 3,  1 | 2 | 4,  1 | 3 4,  2 | 3 4 $&3&19&6& & \\
$1 | 2 | 3,  1 | 2 | 4,  1 | 2 3 4 $&2&18&6& & \\
$1 | 2 | 3,  1 | 2 | 4,  1 | 3 | 4 $&3&20&6& & x\\

$1 | 2 | 3 4 $&2&20&5& x & \\
$1 | 2 | 3,  1 | 2 | 4,  1 | 3 | 4,  2 | 3 | 4 $&3&23&5& & x\\
$1 | 2 | 3,  1 | 2 | 4,  1 | 2 3 4,  1 | 3 | 4 $&3&25&5& & \\
$1 | 2 | 3 | 4 $&2&24&4&x & x\\
\end{longtable}

\bibliography{bibMargInd.bib}{}
\bibliographystyle{plain}

\end{document}